\documentclass[11pt]{article}

\usepackage[T1]{fontenc}
\usepackage{lmodern}
\usepackage[utf8]{inputenc}
\usepackage{microtype}
\usepackage{framed}
\usepackage{listings}
\usepackage{vmargin}
\usepackage{setspace}
\usepackage{mathrsfs, mathenv}
\usepackage{amsmath, amsthm, amssymb, amsfonts, amscd}
\usepackage{graphicx}
\usepackage{epstopdf}
\usepackage[svgnames]{xcolor}
\usepackage{hyperref}
\hypersetup{citecolor=blue, colorlinks=true, linkcolor=black}
\usepackage[ruled, vlined]{algorithm2e}
\usepackage[capitalise]{cleveref}
\usepackage{subcaption}
\usepackage{bbm}
\usepackage{cite}
\usepackage{verbatim}
\usepackage{pgfplots}
\usepackage{tikz}
\usetikzlibrary{arrows,decorations.pathmorphing,backgrounds,positioning,fit,matrix}
\usepackage{etoolbox}
\usepackage{color}
\usepackage{lipsum}
\usepackage{ifthen}
\usepackage[title]{appendix}
\usepackage{autonum}
\usepackage{accents}
\usepackage{xpatch}
\usepackage[]{algpseudocode}
\usepackage{multicol}
\usepackage[abs]{overpic}
\usepackage{eqnarray}

\crefname{assumption}{Assumption}{Assumptions}
\theoremstyle{plain}
\newtheorem{theorem}{Theorem}[section]
\newtheorem{corollary}[theorem]{Corollary}
\newtheorem{lemma}[theorem]{Lemma}
\newtheorem{proposition}[theorem]{Proposition}
\numberwithin{equation}{section}

\theoremstyle{definition}

\theoremstyle{remark}
\newtheorem{remark}[theorem]{Remark}
\newtheorem{assumption}[theorem]{Assumption}
\newtheorem{example}[theorem]{Example}

\ifpdf
  \DeclareGraphicsExtensions{.eps,.pdf,.png,.jpg}
\else
  \DeclareGraphicsExtensions{.eps}
\fi

\usepackage{mathtools}
\usepackage{booktabs}

\usepackage{pgfplots}
\usepackage{tikz}
\usetikzlibrary{patterns,arrows,decorations.pathmorphing,backgrounds,positioning,fit,matrix}
\usepackage[labelfont=bf]{caption}
\usepackage{here}
\usepackage[font=normal]{subcaption}

\usepackage{enumitem}
\setlist[itemize]{leftmargin=.5in}
\setlist[enumerate]{leftmargin=.5in,topsep=3pt,itemsep=3pt,label=(\roman*)}

\newcommand{\email}[1]{\href{#1}{#1}}
\newcommand{\TheTitle}{Eigenfunction martingale estimators for interacting particle systems and their mean field limit}
\newcommand{\TheAuthors}{G. A. Pavliotis, A. Zanoni}
\title{\TheTitle}
\author{Grigorios A. Pavliotis \thanks{Department of Mathematics, Imperial College London, London SW7 2AZ, UK, \email{g.pavliotis@imperial.ac.uk}}
\and Andrea Zanoni \thanks{Institute of Mathematics, École Polytechnique Fédérale de Lausanne, 1015 Lausanne, Switzerland, \email{andrea.zanoni@epfl.ch}}
}
\date{}

\usepackage{amsopn}

\newcommand{\abs}[1]{\left\lvert#1\right\rvert}
\newcommand{\norm}[1]{\left\|#1\right\|}
\renewcommand{\Pr}{\mathbb{P}}

\newcommand{\R}{\mathbb{R}}

\newcommand{\epl}{\varepsilon}

\newcommand{\defeq}{\coloneqq}
\newcommand{\eqdef}{\eqqcolon}

\newcommand{\E}{\operatorname{\mathbb{E}}}

\newcommand{\compl}{\mathsf{C}}

\renewcommand{\d}{\mathrm{d}}
\newcommand{\dd}{\,\mathrm{d}}
\definecolor{shade}{RGB}{100, 100, 100}
\definecolor{bordeaux}{RGB}{128, 0, 50}

\renewcommand*{\dot}[1]{\accentset{\mbox{\large\bfseries .}}{#1}}

\usepackage[usestackEOL]{stackengine}

\definecolor{leg1}{RGB}{0,114,189}
\definecolor{leg2}{RGB}{217,83,25}
\definecolor{leg3}{RGB}{237,177,32}
\definecolor{leg4}{RGB}{126,47,142}
\definecolor{leg5}{RGB}{119,172,48}

\definecolor{leg21}{RGB}{62,38,169}
\definecolor{leg22}{RGB}{46,135,247}
\definecolor{leg23}{RGB}{55,200,151}
\definecolor{leg24}{RGB}{254,195,56}

\ifpdf
\hypersetup{
	pdftitle={\TheTitle},
	pdfauthor={\TheAuthors}
}
\fi

\begin{document}
	
\maketitle	

\begin{abstract} 
We study the problem of parameter estimation for large exchangeable interacting particle systems when a sample of discrete observations from a single particle is known. We propose a novel method based on martingale estimating functions constructed by employing the eigenvalues and eigenfunctions of the generator of the mean field limit, where the law of the process is replaced by the (unique) invariant measure of the mean field dynamics. We then prove that our estimator is asymptotically unbiased and asymptotically normal when the number of observations and the number of particles tend to infinity, and we provide a rate of convergence towards the exact value of the parameters. Finally, we present several numerical experiments which show the accuracy of our estimator and corroborate our theoretical findings, even in the case the mean field dynamics exhibit more than one steady states.
\end{abstract}



\section{Introduction}

Interacting particle systems and, more generally interacting multiagent models, appear frequently in the natural and social sciences. In addition to the well known applications, e.g., plasma physics~\cite{Gol16} and stellar dynamics~\cite{BiT08}, new applications include, e.g., the modeling of chemotaxis~\cite{Suz05}, pedestrian dynamics~\cite{GSW19,MKT18}, crowd dynamics~\cite{MaF19}, urban modeling~\cite{EGP18}, models for opinion formation~\cite{GPY17,GGS21}, collective behavior~\cite{Daw83}, and models for systemic risk~\cite{GSS20}. In many of these applications, the phenomenological models involve unknown parameters that need to be estimated from data. This is particularly the case for multiagent models used in the social sciences and in economics, where no physics-informed choices of parameters are available. Learning parameters or even models, in a nonparametric setting, from data is becoming an increasingly important aspect of the overall mathematical modeling strategy. This is particularly the case in view of the huge quantity of available data in different areas, which allows the development of accurate data-driven techniques for learning parameters from data.

In this paper we study the problem of inference for systems of (weakly) interacting diffusions for which the mean field limit exists and is described by a nonlinear diffusion process of McKean type, obtained in the limit as the number of interacting processes $N$ goes to infinity. When the number of interacting stochastic differential equations (SDEs) is large, the inference problem can become computationally intractable and it is often useful to study the problem of parameter estimation for the limiting mean field SDE. This is related, but distinct, from the problem of inference for multiscale diffusions~\cite{AGP21,APZ21,GaZ21,PPS09,PaS07} where the objective is to learn the parameters in the homogenized (limiting) SDE from observations of the full dynamics. Our goal is to show how the inference methodology using eigenfunction martingale estimating functions that was applied in~\cite{APZ21} to multiscale diffusions can be modified so that it can also be applied to interacting diffusions with a well defined mean field limit. It is useful to keep in mind the analogy between the homogenization and mean field limits, in the context of parameter estimation.

Inference for large interacting systems has attracted considerable attention, starting from the work of Kasonga~\cite{Kas90}, in which the maximum likelihood estimator (MLE) was considered. In particular, it was proved that the MLE for estimating parameters in the drift, when the drift is linearly dependent on the parameters, given continuous time observations of \emph{all} the particles of the $N$-particle system, is consistent and asymptotically normal in the limit as $N \to \infty$. In this setting, it is possible to test whether the particles are interacting or not, at least in the linear case, i.e., for a system of interacting Ornstein--Uhlenbeck processes. Consistency and asymptotic normality of the sieve estimator and an approximate MLE estimator, i.e., when discrete observations of all the particles are given, was studied in~\cite{Bis11} in the same framework of linear dependence on the parameters for the drift and known diffusion coefficient. Moreover, MLE inference of the mean field Ornstein--Uhlenbeck SDE was also considered. Properties of the MLE for the McKean SDE, when a continuous path of the SDE is observed, were studied in~\cite{WWM16}. Consistency of the MLE was proved and an application to a model for ionic diffusion was presented. The MLE estimator for the McKean SDE was also considered in~\cite{LiQ20} and numerical experiments for the mean field Ornstein--Uhlenbeck process were presented. The combined large particle and long time asymptotics, $N \to \infty$ and $T \to \infty$, of the MLE for the case of a quadratic interaction, i.e., for interacting Ornstein--Uhlenbeck processes, was studied in~\cite{Che21}. Unlike the previous works mentioned in this literature review, the case where only a single particle trajectory is observed was considered in this paper. It was shown that the parameters in the drift can be estimated with optimal rate of convergence simultaneously in mean-field limit and in long-time dynamics. Offline and online inference for the McKean SDE was studied in~\cite{SKP21}. Consistency and asymptotic normality of the offline MLE for the interacting particle system in the limit as the number of particles $N \to \infty$ was shown. In addition, an online parameter estimator for the mean field SDE was proposed, which evolves according to a continuous-time stochastic gradient descent algorithm on the asymptotic log-likelihood of the interacting particle system.

In this paper we consider systems of exchangeable weakly interacting diffusions for which uniform propagation of chaos results are known~\cite{BRT98,BRV98,DeT21,Mal01,Oel84} and for which the mean field SDE has a unique invariant measure. We assume that we are given a sample of discrete-time observations of a single particle. Due to exchangeability, this amount of information should be sufficient to infer parameters in the mean field SDE, in the joint asymptotic limit as the number of observations and the number of particles go to infinity. Our approach consists of constructing martingale estimating functions~\cite{BiS95,KeS99} based on the eigenvalues and the eigenfunctions of the generator of the mean field dynamics. Then, our eigenfunction estimator is the zero of the estimating function. The martingale estimator based on the eigenfunctions of the generator was used to study the inference problem for multiscale diffusions in~\cite{APZ21}. Unlike the finite dimensional case, the mean field SDE is a measure-valued process and the generator is a nonlinear operator, dependent on the law of the process. A direct application of the martingale eigenfunction estimator would require the solution of a nonlinear eigenvalue problem that can be computationally demanding and that would also lead to eigenfunctions depending on time via their dependence on the law of the process. We circumvent this difficulty by replacing the law of the process with the (unique) invariant measure of the mean field dynamics. This leads to a standard Sturm--Liouville type of eigenvalue problem that we can analyze and also solve numerically at a low computational cost. In this paper we consider the framework where the invariant measure of the mean field SDE is unique. We remark, however, that our numerical experiments show that our methodology applies to McKean SDEs that exhibit phase transitions, i.e., that have multiple stationary measures, as long as we are below the transition point, or the form of the invariant measure is known up to a finite set of parameters, e.g., moments. 

When the mean field dynamics has a unique invariant measure, we first show the existence of the estimator with high probability when the number of available data and particles is large enough, and then analyze its consistency proving the asymptotic convergence towards the true value of the unknown parameter and providing a rate. Moreover, we prove that the estimator is  asymptotically normal. We also note that the relationship between the number of observations and particles plays an important role in the study of the asymptotic properties of the estimator, in particular the latter must be sufficiently greater than the former in order for the previous results to hold. We then present a series of numerical experiments which confirm our theoretical results and we show the advantages of our method with respect to the MLE. In particular, in contrast with our estimator, the MLE is biased when we have sparse observations, i.e., when the sampling rate $\Delta$ is far from the asymptotic limit $\Delta \to 0$. 

\paragraph{Main contributions.} The main contributions of our work are summarized below.
\begin{itemize}[leftmargin=0.5cm]
\item We propose a new methodology for estimating parameters in the drift of large interacting particle systems when a sequence of discrete observations of a single particle is given. Our proposed estimator is based on the eigenvalues and eigenfunctions of the generator of the mean field SDE at the steady state.
\item We show theoretically that our estimator is asymptotically unbiased and asymptotically normal in the limit as the number of observations and the number of particles go to infinity and we compute the rate of convergence.
\item We demonstrate numerically that our proposed estimator is reliable and robust with respect to the sampling rate.
\end{itemize}

\paragraph{Outline.} The rest of the paper is organized as follows. In \cref{sec:problem_setting} we introduce the framework of the problem under investigation and we present the main theoretical results, and in \cref{sec:numerical_experiments} we show several numerical experiments illustrating the potentiality of our approach. Finally, \cref{sec:proofs} is devoted to the proofs of the main theorems.

\section{Problem setting} \label{sec:problem_setting}

In this work we consider a system of interacting particles in one dimension moving in a confining potential over the time interval $[0,T]$ whose interaction is governed by an interaction potential
\begin{equation} \label{eq:SDE_N}
\begin{aligned}
\dd X_t^{(n)} &= - V'(X_t^{(n)}; \alpha) \dd t - \frac{1}{N} \sum_{i=1}^N W'(X_t^{(n)} - X_t^{(i)}; \kappa) \dd t + \sqrt{2\sigma} \dd B_t^{(n)}, \qquad n = 1, \dots, N, \\
X_0^{(n)} &\sim \nu, \qquad n = 1, \dots, N,
\end{aligned}
\end{equation}
where $N$ is the number of particles, $\{ B_t^{(n)} \}_{n=1}^N$ are standard independent one dimensional Brownian motions, $V(\cdot; \alpha)$ and $W(\cdot; \kappa)$ are the confining and interaction potentials, respectively, which depend on some parameters $\alpha \in \R^{p_1}, \kappa \in \R^{p_2}$, and $\sigma > 0$ is the diffusion coefficient. The functions $V'$ and $W'$ are then the derivatives of $V$ and $W$ with respect to their first argument. We assume chaotic initial conditions, i.e., that the particles are initially distributed according to the same measure $\nu$.

\begin{remark} \label{rem:dimension}
We consider the case when the particles move in one dimension for the clarity of exposition. In fact, the proposed method and our rigorous results can be easily generalized to the case of $N$ interacting particles moving in dimension $d > 1$. However in higher dimensions the problem becomes more complex and expensive from the computational point of view.
\end{remark}

We place ourselves in the same framework of \cite{Mal01}, which is summarized in the following assumption.

\begin{assumption} \label{ass:potential}
The confining and interaction potentials $V$ and $W$, respectively, satisfy:
\begin{itemize}[leftmargin=0.5cm]
\item $V(\cdot; \alpha) \in \mathcal C^2(\R)$ is uniformly convex and polynomially bounded along with its derivatives uniformly in $\alpha$;
\item $W(\cdot; \kappa) \in \mathcal C^2(\R)$ is even, convex and polynomially bounded along with its derivatives uniformly in $\kappa$.
\end{itemize}
\end{assumption}

It is well-known (see, e.g., \cite[Chapter 4]{Pav14}) that under \cref{ass:potential} the dynamics described by the system \eqref{eq:SDE_N} is geometrically ergodic with unique invariant measure given by the Gibbs measure $\mu^N_\theta(\dd \mathbf x) = \rho^N(\mathbf x; \theta) \dd \mathbf x$, where 
\begin{equation}
\rho^N(\mathbf x; \theta) = \frac{1}{Z^N} \exp \left\{- \frac1\sigma E^N(\mathbf x; \theta) \right\}, \qquad Z^N = \int_{\R^N} \exp \left\{- \frac1\sigma E^N(\mathbf x; \theta) \right\} \dd \mathbf x,
\end{equation}
and $E^N(\cdot; \theta)$ is defined by
\begin{equation}
E^N(\mathbf x; \theta) \defeq \sum_{n=1}^N V(x_n; \alpha) + \frac{1}{2N} \sum_{n=1}^N \sum_{i=1}^N W(x_n - x_i; \kappa).
\end{equation}
for $\theta = \begin{pmatrix} \alpha^\top & \kappa^\top \end{pmatrix}^\top \in \Theta \subseteq \R^p$ with $p = p_1 + p_2$ and $\Theta$ the set of admissible parameters. The main goal of this paper is the estimation of the unknown parameter $\theta \in \Theta$, given discrete observations of the path of one single particle. We are interested in applications involving large interacting particle systems, i.e., when $N \gg 1$, hence studying the whole system is not practical and can be computationally unfeasible. Therefore, our approach consists of considering the mean field limit which has already been thoroughly studied (see, e.g., \cite{Daw83,Gar88}). Letting the number of particles $N$ go to infinity we obtain the nonlinear, in the sense of McKean, SDE
\begin{equation} \label{eq:SDE_MFL}
\begin{aligned}
\dd X_t &= - V'(X_t; \alpha) \dd t - (W'(\cdot; \kappa) * u(\cdot, t; \theta)) (X_t) \dd t + \sqrt{2\sigma} \dd B_t, \\
X_0 &\sim \nu,
\end{aligned}
\end{equation}
where $u(\cdot, t; \theta)$ is the density with respect to the Lebesgue measure of the law of $X_t$ and the nonlinearity means that the drift of the SDE \eqref{eq:SDE_MFL} depends on the law of the process. The density $u$ is the solution of the nonlinear Fokker--Planck (McKean--Vlasov) equation
\begin{equation} \label{eq:McKeanVlasov}
\frac{\partial u}{\partial t} (x,t) = \frac{\partial}{\partial x} \left( V'(x;\alpha) u(x,t; \theta) + (W'(\cdot; \kappa) * u(\cdot, t; \theta))(x,t) u(x,t; \theta) + \sigma \frac{\partial u}{\partial x}(x,t) \right),
\end{equation}
with initial condition $u(x,0; \theta) \dd x = \nu(\dd x)$. It is well known that, in contrast to the finite dimensional dynamics, the mean field limit \eqref{eq:SDE_MFL} can have, in the non-convex case more than one invariant measures $\mu_\theta(\dd x ) = \rho(x; \theta) \dd x$ \cite{CGP20,Daw83}. The density of the stationary state(s) satisfies the stationary Fokker--Planck equation
\begin{equation} \label{eq:McKeanVlasov_stationary}
\frac{\d}{\dd x} \left( V'(x;\alpha) \rho(x; \theta) + (W'(\cdot; \kappa) * \rho(\cdot; \theta))(x) \rho(x; \theta) + \rho'(x; \theta) \right) = 0,
\end{equation}
where the second variable $\theta$ emphasizes the fact that $\rho$ depends on the parameters $\alpha$ and $\kappa$ of the potentials $V$ and $W$, respectively. However, under \cref{ass:potential} it has been proven in \cite{Mal01} that there exists a unique invariant measure which is the solution of 
\begin{equation} \label{eq:invariant density}
\rho(x;\theta) = \frac1Z \exp \left\{-\frac1\sigma \left( V(x;\alpha) + (W(\cdot;\kappa) * \rho(\cdot; \theta))(x) \right) \right\},
\end{equation}
where $Z$ is the normalization constant 
\begin{equation}
Z = \int_\R \exp \left\{-\frac1\sigma \left( V(x;\alpha) + (W(\cdot;\kappa) * \rho(\cdot; \theta))(x) \right) \right\} \dd x.
\end{equation}

\begin{example} \label{exa:CW}
A particular choice for the interaction potential is the Curie--Weiss quadratic interaction \cite{Daw83}, which is also known as harmonic potential. We take $\kappa > 0$ and consider the confining potential
\begin{equation}
W(x;\kappa) = \frac\kappa2 x^2.
\end{equation}
The interacting particles system \eqref{eq:SDE_N} becomes, for all $n = 1, \dots, N$
\begin{equation} \label{eq:SDE_N_CW}
\dd X_t^{(n)} = - V'(X_t^{(n)}; \alpha) \dd t - \kappa \left( X_t^{(n)} - \bar X_t^N \right) \dd t + \sqrt{2\sigma} \dd B_t^{(n)},
\end{equation}
where $\bar X_t^N$ denotes the empirical mean
\begin{equation}
\bar X_t^N = \frac{1}{N} \sum_{i=1}^N X_t^{(i)}.
\end{equation}
This interaction term creates a tendency for the particles to relax toward the center of gravity of the ensemble and the parameter $\kappa$ measures the strength of the interaction between the agents, hence this model provides a simple example of cooperative interaction. 

The mean field limit \eqref{eq:SDE_MFL} then becomes
\begin{equation} \label{eq:SDE_MFL_CW}
\dd X_t = - V'(X_t; \alpha) \dd t - \kappa \left( X_t - m_t \right) \dd t + \sqrt{2\sigma} \dd B_t,
\end{equation}
where $m_t$ denotes the expectation of $X_t$, $m_t = \E[X_t]$, and its unique (when the confining potential $V$ is convex) invariant measure $\mu_\theta(\dd x) = \rho(x; \theta) \dd x$ is given by
\begin{equation} \label{eq:rho_CW}
\rho(x; \theta) = \frac1Z \exp \left\{-\frac1\sigma \left( V(x; \alpha) + \kappa \left( \frac12 x^2 - mx \right) \right) \right\},
\end{equation}
with the constraint for the expectation with respect to the invariant measure
\begin{equation}\label{e:self_consist}
m = \int_\R x \rho(x; \theta) \dd x,
\end{equation}
and where
\begin{equation}
Z = \int_\R \exp \left\{ -\frac1\sigma \left( V(x; \alpha) + \kappa \left( \frac12 x^2 - mx \right) \right) \right\} \dd x.
\end{equation}
Equation \eqref{e:self_consist} is the self-consistency equation \cite{Daw83,Fra05,GoP18} that enables us to calculate the invariant measure and, then, the stationary state(s). In the case where the confining potential is quadratic, we have a system linear SDEs and the mean field limit reduces to the mean field Ornstein-Uhlenbeck SDE. In this case the first moment vanishes, $m = 0$, and the invariant measure is unique (this is the case, of course, of arbitrary strictly convex confining potentials). The inference problem for the linear interacting particle system and for the corresponding mean field limit is easier than that of the general case. We emphasize that, unlike this present work, most earlier papers, e.g., \cite{Bis11,Kas90}, focus on this linear case, i.e., on systems of weakly interacting linear stochastic differential equations. The estimator proposed and studied in this paper can be applied to arbitrary non-quadratic interaction and confining potentials.
\end{example}

\subsection{Parameter estimation problem}

We now present our method for the estimation of the unknown parameter $\theta = (\alpha, \kappa) \in \Theta \subseteq \R^p$, given discrete observation of a single particle of the system \eqref{eq:SDE_N}. Consider $M+1$ equidistant observation times $0 = t_0 < t_1 < \dots < t_M = T$, let $\Delta = t_m - t_{ m-1}$ be the sampling rate and let $(X_t^{(n)})_{t \in [0,T]}$ be a realization of the $n$-th particle of the solution of the system \eqref{eq:SDE_N} for some $n = 1, \dots, N$. We then aim to estimate the unknown parameter $\theta$ given a sample $\{ \widetilde X_m^{(n)} \}_{m=0}^M$ of the realization where $\widetilde X_m^{(n)} = X^{(n)}_{t_m}$ and $t_m = \Delta m$. We want to construct martingale estimating functions based on the eigenfunctions and the eigenvalues of the generator of the dynamics, a technique which was initially proposed in \cite{KeS99} for single-scale SDEs and then successfully applied to multiscale SDEs in \cite{APZ21}. In principle, the methodology developed in \cite{KeS99} can be applied to the $N-$particle system. However, this would require solving the eigenvalue problem for the generator of an $N-$dimensional diffusion process, which is computationally expensive. Moreover, our fundamental assumption is that are observing a single particle and thus we do not have a complete knowledge of the system. Therefore, we construct the martingale estimating functions employing the generator of the mean field dynamics, which is a good approximation of the path of a single particle when the number $N$ of particles is large \cite{Szn91}. Let $\mathcal L_t$ be the generator of the mean field limit SDE \eqref{eq:SDE_MFL}
\begin{equation}
\mathcal L_t = - \left( V'(\cdot; \alpha) + (W'(\cdot; \kappa) * u(\cdot, t; \theta)) \right) \frac{\d}{\dd x} + \sigma \frac{\d^2}{\dd x^2},
\end{equation}
and let $\mathcal L$ be the generator obtained replacing the density $u(\cdot, t; \theta)$ with the density $\rho(\cdot; \theta)$ of the invariant measure $\mu_\theta$
\begin{equation}
\mathcal L = - \left( V'(\cdot; \alpha) + (W'(\cdot; \kappa) * \rho(\cdot; \theta)) \right) \frac{\d}{\dd x} + \sigma \frac{\d^2}{\dd x^2}.
\end{equation}
We remark that now the generator $\mathcal L$ is time-independent. We then consider the eigenvalue problem $- \mathcal L \phi(\cdot;\theta) = \lambda(\theta) \phi(\cdot;\theta)$, which reads
\begin{equation} \label{eq:eigenvalue_problem_equation}
\sigma \phi''(x;\theta) - \left( V'(x; \alpha) + (W'(\cdot; \kappa) * \rho(\cdot; \theta))(x) \right) \phi'(x;\theta) + \lambda(\theta) \phi(x;\theta) = 0,
\end{equation}
and from the well-known spectral theory of diffusion processes (see, e.g., \cite{HST98}) we deduce the existence of a countable set of eigenvalues $0 = \lambda_0(\theta) < \lambda_1(\theta) < \cdots < \lambda_j(\theta) \uparrow \infty$ whose corresponding eigenfunctions $\{ \phi_j(\cdot; \theta) \}_{j = 0}^\infty$ form an orthonormal basis of the weighted space $L^2(\rho(\cdot; \theta))$. In fact, even if the SDE \eqref{eq:SDE_MFL} is nonlinear, when $X_0 \sim \rho(\cdot; \theta)$ then the solution $X_t$ behaves like a classic diffusion process with drift function $- V'(\cdot; \alpha) - W'(\cdot; \kappa) * \rho(\cdot; \theta)$, hence the spectral theory for diffusion processes still holds. We also state here the variational formulation of the eigenvalue problem, which will be employed to implement numerically the proposed methodology. Let $\varphi$ be a test function and multiply equation \eqref{eq:eigenvalue_problem_equation} by $\varphi \rho(\cdot;\theta)$, where the density $\rho(\cdot;\theta)$ of the invariant measure $\mu_\theta$ is defined in \eqref{eq:invariant density}. Then, integrating over $\R$ and by parts we obtain
\begin{equation}
\sigma \int_\R \phi'(x;\theta) \varphi'(x) \rho(x;\theta) \dd x = \lambda(\theta) \int_\R \phi(x;\theta) \varphi(x) \rho(x;\theta) \dd x.
\end{equation}
We are now ready to present how to employ the eigenvalue problem in the construction of the martingale estimation function and afterwords in the definition of our estimator. Let $J$ be a positive integer and let $\psi_j(\cdot; \theta) \colon \R \to \R^p$ for $j = 1, \dots, J$ be arbitrary functions dependent on the parameter $\theta$ which satisfy \cref{ass:functions_psi} below, and define the martingale estimating function $G_{M,N}^J \colon \Theta \to \R^p$ as
\begin{equation}
G_{M,N}^J(\theta) \defeq \frac1M \sum_{m=0}^{M-1} \sum_{j=1}^J g_j(\widetilde X_m^{(n)}, \widetilde X_{m+1}^{(n)}; \theta),
\end{equation} 
where 
\begin{equation} \label{eq:def_g}
g_j(x,y;\theta) \defeq \psi_j(x;\theta) \left( \phi_j(y;\theta) - e^{-\lambda_j(\theta) \Delta} \phi_j(x;\theta) \right),
\end{equation}
and $\{ \widetilde X_m^{(n)} \}_{m=0}^M$ is the set of observations of the $n$-th particle from the system with $N$ particles. The estimator we propose is then given by the solution $\widehat \theta_{M,N}^J$ of the $p$-dimensional nonlinear system 
\begin{equation} \label{eq:system_zero}
G_{M,N}^J(\theta) = \mathbf 0,
\end{equation}
where $\mathbf 0 \in \R^p$ denotes the vector with all components equal to zero. An intuition on why considering the solution of equation (2.8) as a good estimator is the following and will be more clear later. Let $\mathbb G^J_M$ defined in \eqref{eq:def_Hfunctions} be the estimating function where the observations from the inetracting particle system have been replaced by the observations from the corresponding mean field limit. Then, employing formula \eqref{eq:formula_martingale} we have
\begin{equation}
\E^{\mu_{\theta_0}} \left[ \mathbb G^J_M(\theta_0) \right] = 0,
\end{equation}
which means that the zero of the expectation of the estimating function with observations from the mean field limit is exactly the true unknown coefficient. The main steps needed to obtain the estimator $\widehat \theta_{M,N}^J$ are summarized in \cref{alg:procedure}. For further details about the implementation and for discussions about the choice of the arbitrary functions $\{ \psi_j(\cdot;\theta) \}_{j=1}^J$ we refer to Appendix B and Remark 2.6 in \cite{APZ21}.

\begin{remark} \label{rem:knowledge_invariant_measure}
The main limitation of our approach is that the knowledge of the invariant measure is required in order to construct the martingale estimating function (step 1 in \cref{alg:procedure}). However, it is often the case that the invariant measure is known up to a set of parameters, such as moments, i.e., only the functional form of the invariant measure is known. These parameters (moments) are obtained by solving appropriate self-consistency equations \cite[Section 2.3]{Fra05}. When such a situation arises, it is possible to first learn these parameters using the available data, e.g., estimate the moments that appear in the invariant measure by employing the law of large numbers. Then, we are in the setting where our technique applies and we can proceed in the same way, as shown in the numerical experiments in \cref{sec:bistable,sec:num_nonsymmetric}. In summary, it is sufficient to replace step 1 in \cref{alg:procedure} with ``estimate the moments in the invariant measure $\rho(\cdot;\theta)$''.
\end{remark}

We finally introduce a technical hypothesis which will be needed for the proofs of our main results.

\begin{assumption} \label{ass:functions_psi}
Let $\Theta \subseteq \R^p$ be a compact set. Then the following hold for all $\theta \in \Theta$ and for all $j = 1, \dots, J$:
\begin{enumerate}
\item $\psi_j(x;\theta)$ is continuously differentiable with respect to $\theta$ for all $x \in \R$;
\item all components of $\psi_j(\cdot;\theta)$, $\psi_j'(\cdot;\theta)$, $\dot{\psi}_j(\cdot;\theta)$, $\dot \psi_j'(\cdot;\theta)$ are polynomially bounded uniformly in $\theta$;
\item the potentials $V$ and $W$ are such that $\phi_j(\cdot;\theta)$, $\phi_j'(\cdot;\theta)$ and all components of $\dot \phi_j(\cdot;\theta)$, $\dot \phi_j'(\cdot;\theta)$ are polynomially bounded uniformly in $\theta$;
\end{enumerate}
where the dot denotes either the Jacobian matrix or the gradient with respect to $\theta$.
\end{assumption}

\begin{remark} \label{rem:continuityG}
\cref{ass:functions_psi}(i) together with \cite[Sections 2 and 6]{Sch74} gives the continuous differentiability of the vector-valued function $G_{M,N}^J(\theta)$ with respect to the unknown parameter $\theta$.
\end{remark}

\begin{algorithm}
\caption{Estimation of $\theta \in \Theta$} \label{alg:procedure}
\begin{tabbing}
\hspace*{\algorithmicindent} \textbf{Input:} \= Observations $\{ \widetilde X_m^{(n)} \}_{m=0}^M$. \\
\> Distance between two consecutive observations $\Delta$. \\
\> Number of eigenvalues and eigenfunctions $J$. \\
\> Functions $\left\{ \psi_j(x;\theta) \right\}_{j=1}^J$. \\
\> Confining potential $V$ and interaction potential $W$. \\
\> Diffusion coefficient $\sigma$.
\end{tabbing}
\begin{tabbing}
\hspace*{\algorithmicindent} \textbf{Output:} \= Estimation $\widehat \theta_{M,N}^J$ of $\theta$. \\
\end{tabbing}
\begin{enumerate}[label=\arabic*:,itemsep=5pt]
\item Find the invariant measure $\rho(\cdot; \theta)$.
\item Consider the equation \\ $\sigma \phi''(x;\theta) - \left( V'(x; \alpha) + (W'(\cdot; \kappa) * \rho(\cdot; \theta))(x) \right) \phi'(x;\theta) + \lambda(\theta) \phi(x;\theta) = 0$.
\item Compute the first $J$ eigenvalues $\left\{ \lambda_j(\theta) \right\}_{j=1}^J$ and eigenfunctions $\left\{ \phi_j(\cdot;\theta) \right\}_{j=1}^J$.
\item Construct the function $g_j(x,y;\theta) = \psi_j(x;\theta) \left( \phi_j(y;\theta) - e^{-\lambda_j(\theta)\Delta} \phi_j(x;\theta) \right)$.
\item Construct the score function $G_{M,N}^J(\theta) = \frac1M \sum_{m=0}^{M-1} \sum_{j=1}^J g_j(\widetilde X_m^{(n)}, \widetilde X_{m+1}^{(n)}; \theta)$.
\item Let $\widehat \theta_{M,N}^J$ be the solution of the nonlinear system $G_{M,N}^J(\theta) = \mathbf 0$.
\end{enumerate}
\end{algorithm}

\begin{remark} \label{rem:diffusion}
In this paper we always assume that the diffusion coefficient $\sigma$ in \eqref{eq:SDE_N} is known. We remark that this is not an essential limitation of our methodology; in fact, if the diffusion coefficient is also unknown, we can consider the parameter set to be estimated to be $\widetilde \theta = (\theta, \sigma) = (\alpha, \kappa, \sigma) \in \R^{p+1}$ and repeat the same procedure. The estimator is then obtained as the solution of the nonlinear system of dimension $p+1$ corresponding to \eqref{eq:system_zero}. A numerical experiment illustrating this procedure is given in \cref{sec:num_diffusion}. Moreover, our main theoretical results remain valid and the proofs do not need any major changes. Alternatively, if the sampling rate is sufficiently small, it is possible to first estimate the diffusion coefficient using the quadratic variation and then proceed with the methodology proposed in this paper.
\end{remark}

\begin{example} \label{exa:CW_eigen}
Let us consider the Curie--Weiss quadratic interaction introduced in \cref{exa:CW} as well as a quadratic Ornstein--Uhlenbeck confining potential $V(x; \alpha) = \frac12 x^2$. In this case the only unknown parameter is $\kappa$ and the eigenvalue problem \eqref{eq:eigenvalue_problem_equation} reads
\begin{equation} \label{eq:eigenvalue_problem_CWOU}
\sigma \phi''(x;\theta) - (1 + \kappa) x \phi'(x;\theta) + \lambda(\theta) \phi(x;\theta) = 0,
\end{equation}
so that the eigenvalue and eigenfunctions can be computed analytically \cite[Section 3.1]{APZ21}. In particular, the first eigenvalue and eigenfunction are given by $\lambda_1(\theta) = 1 + \kappa$ and $\phi_1(x;\theta) = x$, respectively. Therefore, letting $\psi_1(x;\theta) = x$ we have an explicit expression for our estimator
\begin{equation} \label{eq:estimator_OU}
\widehat \theta_{M,N}^1 = - 1 - \frac1\Delta \log \left( \frac{\sum_{m=0}^{M-1} \widetilde X_m^{(n)} \widetilde X_{m+1}^{(n)}}{\sum_{m=0}^{M-1} (\widetilde X_m^{(n)})^2} \right).
\end{equation}
For additional details regarding the eigenvalue problem \eqref{eq:eigenvalue_problem_CWOU} we refer to \cite[Section 3.1]{APZ21}. We also remark that when the drift coefficient of the Ornstein--Uhlenbeck process is unknown, i.e., if we consider the confining potential $V(x; \alpha) = \frac\alpha2 x^2,$ then the eigenvalue problem reads
\begin{equation}
\sigma \phi''(x;\theta) - (\alpha + \kappa) x \phi'(x;\theta) + \lambda(\theta) \phi(x;\theta) = 0,
\end{equation}
which only depends on the sum $\alpha + \kappa$ and not on the single parameters alone. Therefore, in this case it is not possible to estimate the unknown coefficients $\alpha$ and $\kappa$, but we can only estimate their sum. This is in contrast with the set up in \cite{Kas90}, where \emph{all} the particles are observed in continuous time. When this amount of information is available, it is possible to check whether or not the particles are interacting, i.e., to check whether $\kappa = 0$ or not (see \cite[Section 4]{Kas90}). 
\end{example}

\subsection{Main results}

In this section we present the main theoretical results of this work. In particular, we prove that our estimator $\widehat \theta_{M,N}^J$ is asymptotically unbiased (consistent) and asymptotically normal as the number of observations $M$ and particles $N$ go to infinity and we compute the rate of convergence towards the true value of the parameter, which we denote by $\theta_0$. Part of the proof of the consistency of the estimator, which will be presented in detail in \cref{sec:proofs}, is inspired by our previous work \cite[Section 5]{APZ21}. In this paper we studied the asymptotic properties of a similar estimator for multiscale SDEs letting the number of observations go to infinity and the multiscale parameter vanish. The proofs or our results in the present work also requires us to perform a rigorous asymptotic analysis with respect to two parameters, the number of observations and the number of particles. 

We first define the Jacobian matrix of the function $g_j$ introduced in \eqref{eq:def_g} with respect to the parameter $\theta$, with $\otimes$ denoting the outer product in $\R^p$,
\begin{equation}
\begin{aligned}
h_j(x,y;\theta) &\defeq \dot g_j(x,y;\theta) \\
&= \dot \psi_j(x;\theta) \left( \phi_j(y;\theta) - e^{-\lambda_j(\theta)\Delta}\phi_j(x;\theta) \right) \\
&\quad + \psi_j(x;\theta) \otimes \left( \dot \phi_j(y;\theta) - e^{-\lambda_j(\theta)\Delta} \left( \dot \phi_j(x;\theta) - \Delta \dot \lambda_j(\theta)\phi_j(x,\theta) \right) \right),
\end{aligned}
\end{equation}
as well as the following quantity
\begin{equation}
\ell_{j,k}(x,y;\theta) \defeq \left( \psi_j(x;\theta) \otimes \psi_k(x;\theta) \right) \left( \phi_j(y;\theta) \phi_k(y;\theta) - e^{-(\lambda_j(\theta) + \lambda_k(\theta))\Delta} \phi_j(x;\theta) \phi_k(x;\theta) \right).
\end{equation}
We remark that whenever we write $\E^{\mu_\theta}$ we mean that $X_0 \sim \mu_\theta$ and similarly for the other probability measures. 

We now present our main results. In \cref{thm:main_unbiased} we prove that our estimator is consistent.

\begin{theorem} \label{thm:main_unbiased}
Let $J$ be a positive integer and let $\{ \widetilde X_m^{(n)} \}_{m = 1}^M$ be a set of observations obtained by system \eqref{eq:SDE_N} with true parameter $\theta_0$. Under \cref{ass:potential,ass:functions_psi} and if
\begin{equation} \label{eq:technical_assumption}
\det \left( \sum_{j=1}^J \E^{\mu_{\theta_0}} \left[ h_j(X_0, X_\Delta; \theta_0) \right] \right) \neq 0,
\end{equation}
there exists $N_0 > 0$ such that for all $N > N_0$ an estimator $\widehat \theta_{M,N}^J$, which solves the system $G_{M,N}^J(\theta) = 0$, exists with probability tending to one as $M$ goes to infinity. Moreover, the estimator $\widehat \theta_{M,N}^J$ is asymptotically unbiased, i.e., 
\begin{align}
\lim_{N \to \infty} \lim_{M \to \infty} \widehat \theta_{M,N}^J &= \theta_0, \qquad \text{in probability}, \label{eq:limit_1} \\
\lim_{M \to \infty} \lim_{N \to \infty} \widehat \theta_{M,N}^J &= \theta_0, \qquad \text{in probability}, \label{eq:limit_2}
\end{align}
and if $M = o(N)$
\begin{equation}
\lim_{M,N \to \infty} \widehat \theta_{M,N}^J = \theta_0, \qquad \text{in probability}. \label{eq:limit_3}
\end{equation}
\end{theorem}

Then, in \cref{thm:main_rate} we provide a rate of convergence for our estimator.

\begin{theorem} \label{thm:main_rate}
Let the assumptions of \cref{thm:main_unbiased} hold, and let us introduce the notation
\begin{equation}
\Xi_{M,N}^J \defeq \left(\frac{1}{\sqrt M} + \frac{1}{\sqrt N}\right)^{-1} \norm{\widehat \theta_{M,N}^J - \theta_0}.
\end{equation}
Then, for all $\epl > 0$ there exists $K_\epl > 0$ such that
\begin{align}
\lim_{N \to \infty} \lim_{M \to \infty} \Pr \left( \Xi_{M,N}^J > K_\epl \right) &< \epl, \label{eq:rate_1} \\
\lim_{M \to \infty} \lim_{N \to \infty} \Pr \left( \Xi_{M,N}^J > K_\epl \right) &< \epl, \label{eq:rate_2} 
\end{align}
and if $M = o(\sqrt N)$
\begin{equation}
\lim_{M,N \to \infty} \Pr \left( \Xi_{M,N}^J > K_\epl \right) < \epl. \label{eq:rate_3}
\end{equation}
\end{theorem}

Finally, in \cref{thm:main_normal} we show that our estimator is asymptotically normal.

\begin{theorem} \label{thm:main_normal}
Let the assumptions of \cref{thm:main_unbiased} hold with $M = o(\sqrt N)$. Then, the estimator $\widehat \theta_{M,N}^J$ is asymptotically normal, i.e.,
\begin{equation}
\lim_{M,N \to \infty} \sqrt{M} \left( \widehat \theta_{M,N}^J - \theta_0 \right) = \Lambda^J \sim \mathcal N(\mathbf 0, \Gamma_0^J), \qquad \text{in distribution},
\end{equation}
where
\begin{equation} \label{eq:variance_normal}
\begin{aligned}
\Gamma_0^J = \left( \sum_{j=1}^J \E^{\mu_{\theta_0}} \left[ h_j(X_0, X_\Delta; \theta_0) \right] \right)^{-1} & \left( \sum_{j=1}^J \sum_{k=1}^J \E^{\mu_{\theta_0}} \left[ \ell_{j,k} (X_0, X_\Delta; \theta_0) \right] \right) \\
\times & \left( \sum_{j=1}^J \E^{\mu_{\theta_0}} \left[ h_j(X_0, X_\Delta; \theta_0) \right] \right)^{-\top}.
\end{aligned}
\end{equation}
\end{theorem}

\begin{remark} \label{rem:technical_assumption}
We note that the technical assumption \eqref{eq:technical_assumption} is not a serious limitation of the validity of the theorem; in fact, it is a nondegeneracy hypothesis which holds true in all nonpathological cases and is equivalent to \cite[Condition 4.2(a)]{KeS99} and \cite[Assumption 3.1]{APZ21}. Moreover, it is not necessary to assume that the matrix $\Gamma_0^J$ in \cref{thm:main_normal} is indeed a covariance matrix because, due to the particular form of the estimating function, this follows directly from the central limit theorem as explained in \cite{KeS99}.
\end{remark}

\begin{remark} \label{rem:order_of_limits}
For the proof of the main results, we need to assume that, roughly speaking, the number of particles goes to infinity faster than the number of observations. It is not clear whether this assumption is strictly necessary. We expect that noncommutativity issues between the different distinguished limits may arise in the case where the mean field dynamics exhibits phase transitions, i.e., when the stationary state is not unique, see \cite{DGP21}. We will study the consequences of this noncommutativity due to phase transitions to the performance of our estimator and, more generally, to the inference problem in future work. 
\end{remark}

\section{Numerical experiments} \label{sec:numerical_experiments}

In this section we present a series of numerical experiments to validate our theoretical results and demonstrate the effectiveness of our estimator in estimate unknown drift parameters of interacting particle systems. In order to generate synthetic data we employ the Euler--Maruyama method with a time step $h = 0.01$ to solve numerically system \eqref{eq:SDE_N} and obtain $(X_t^{(n)})_{t \in [0,T]}$ for all $n = 1, \dots, N$. Notice that in order to preserve the exchangeability property of the system it is important to set the same initial condition for all the particles, hence we take $X_0^{(n)} = 0$ for all $n = 1, \dots, N$. We then randomly choose a value $n^* \in \{ 1, \dots, N \}$ and we assume to know a sample $\{ X_m^{(n^*)} \}_{m=0}^M$ of observations obtained from the $n^*$-th particle with sampling rate $\Delta$. We remark that the parameters $h$ and $\Delta$ are not related to each other, in fact the former is only used to generate the data, while the latter is the actual distance between two consecutive observations. We repeat the same procedure for $L = 5$ different realizations of the Brownian motions and then we compute the average of the values obtained employing our estimator $\widehat \theta_{M,N}^J$. In the following, we first perform a sensitivity analysis with respect to the number of observations $M$, particles $N$ and eigenvalues and eigenfunctions employed in the estimation $J$, then we confirm our theoretical results given in \cref{thm:main_unbiased,thm:main_rate,thm:main_normal} and finally we test our technique with more challenging academic examples which do not exactly fit into the theory.

\subsection{Sensitivity analysis and rate of convergence} \label{sec:num_sensitivity}

\begin{figure}
\centering
\includegraphics[]{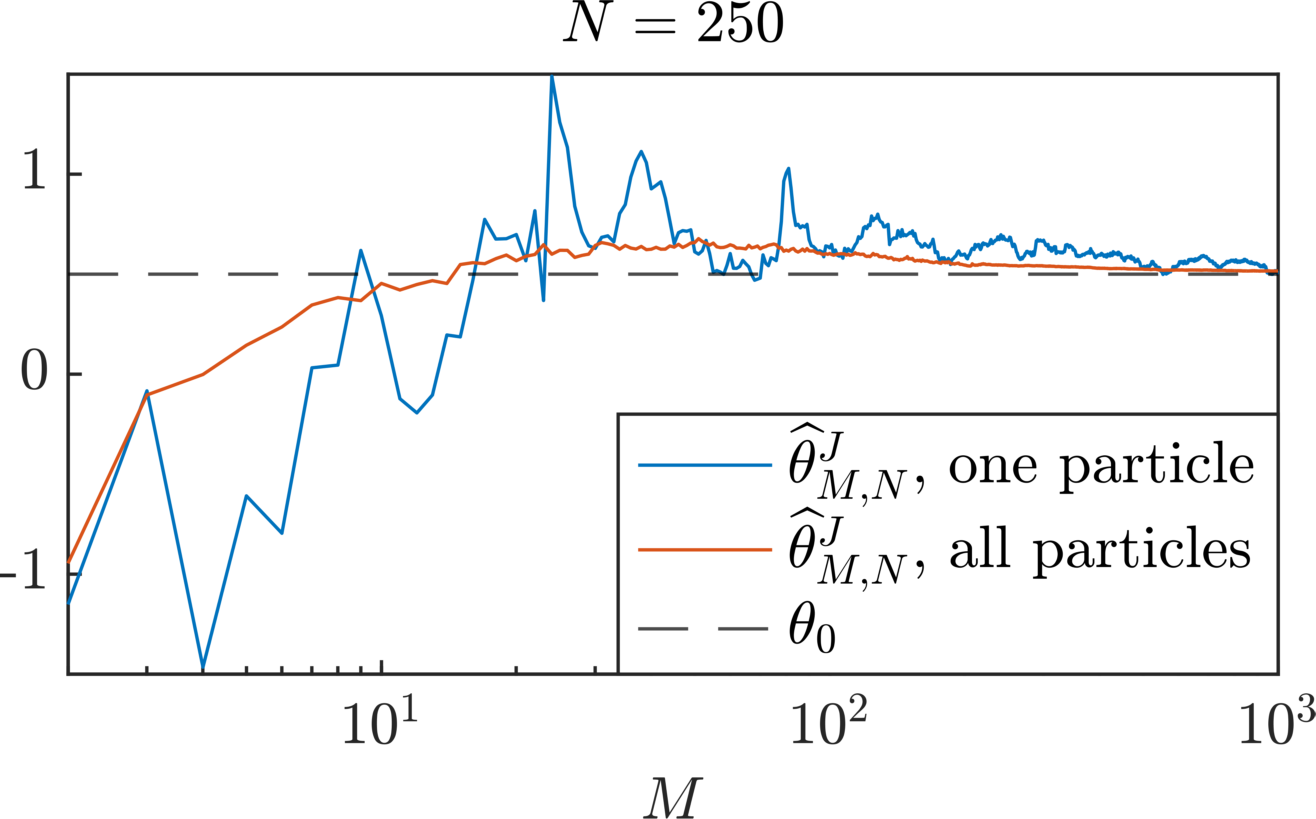} $\quad$
\includegraphics[]{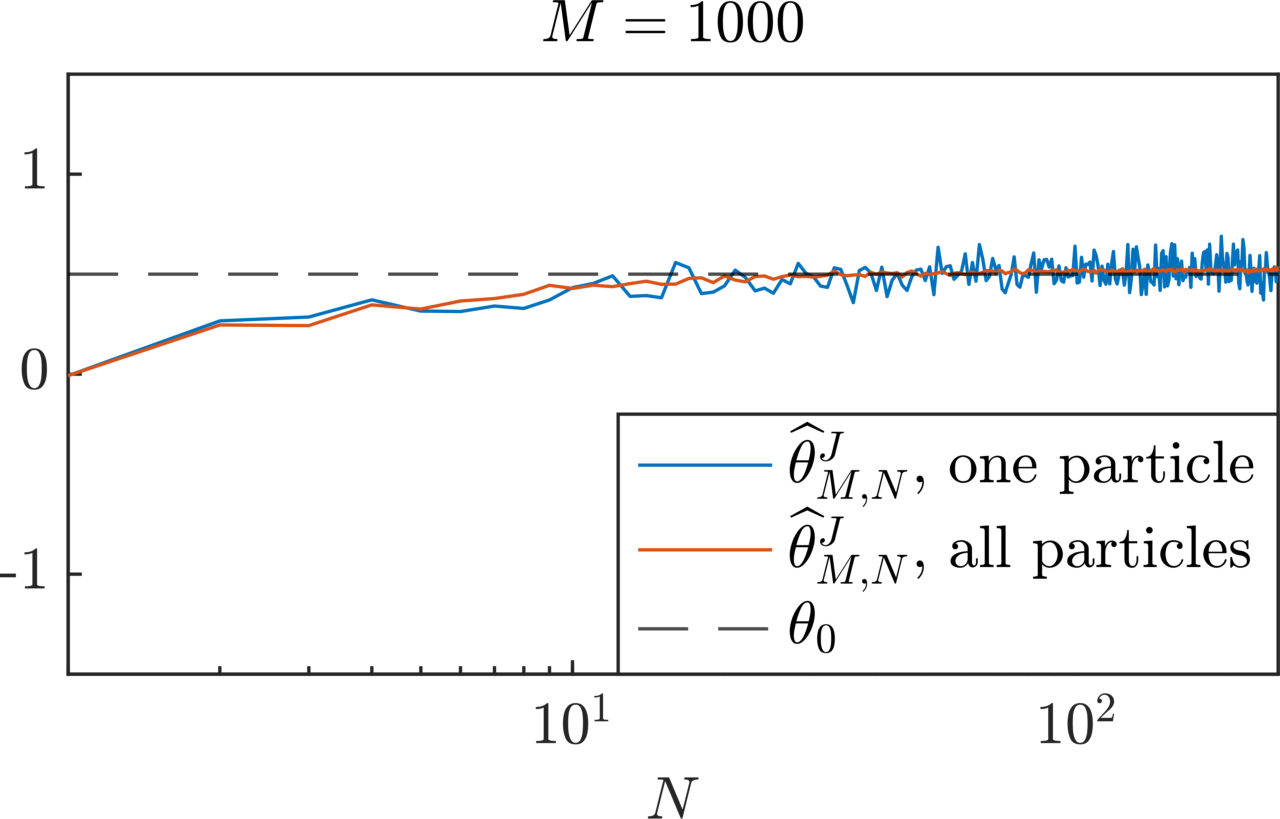} 
\caption{Sensitivity analysis for the Ornstein--Uhlenbeck potential with respect to the number $M$ of observations and $N$ of particles, for the estimator $\widehat \theta^J_{M,N}$ with $J=1$.}
\label{fig:sensitivity_OU_MN}
\end{figure}

\begin{figure}
\centering
\includegraphics[]{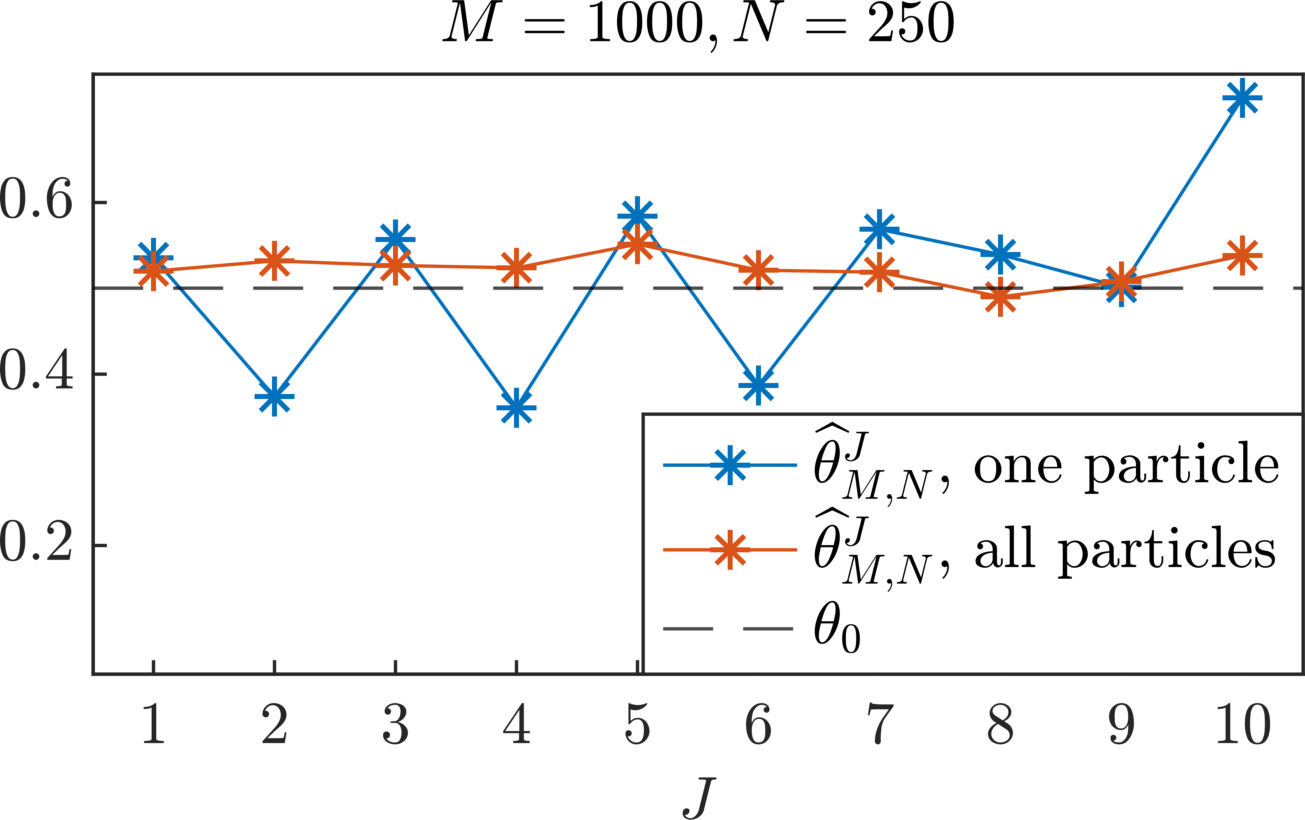}
\caption{Sensitivity analysis for the Ornstein--Uhlenbeck potential with respect to the number $J$ of eigenvalues and eigenfunctions, for the estimator $\widehat \theta^J_{M,N}$.}
\label{fig:sensitivity_OU_J}
\end{figure}

\begin{figure}
\centering
\includegraphics[]{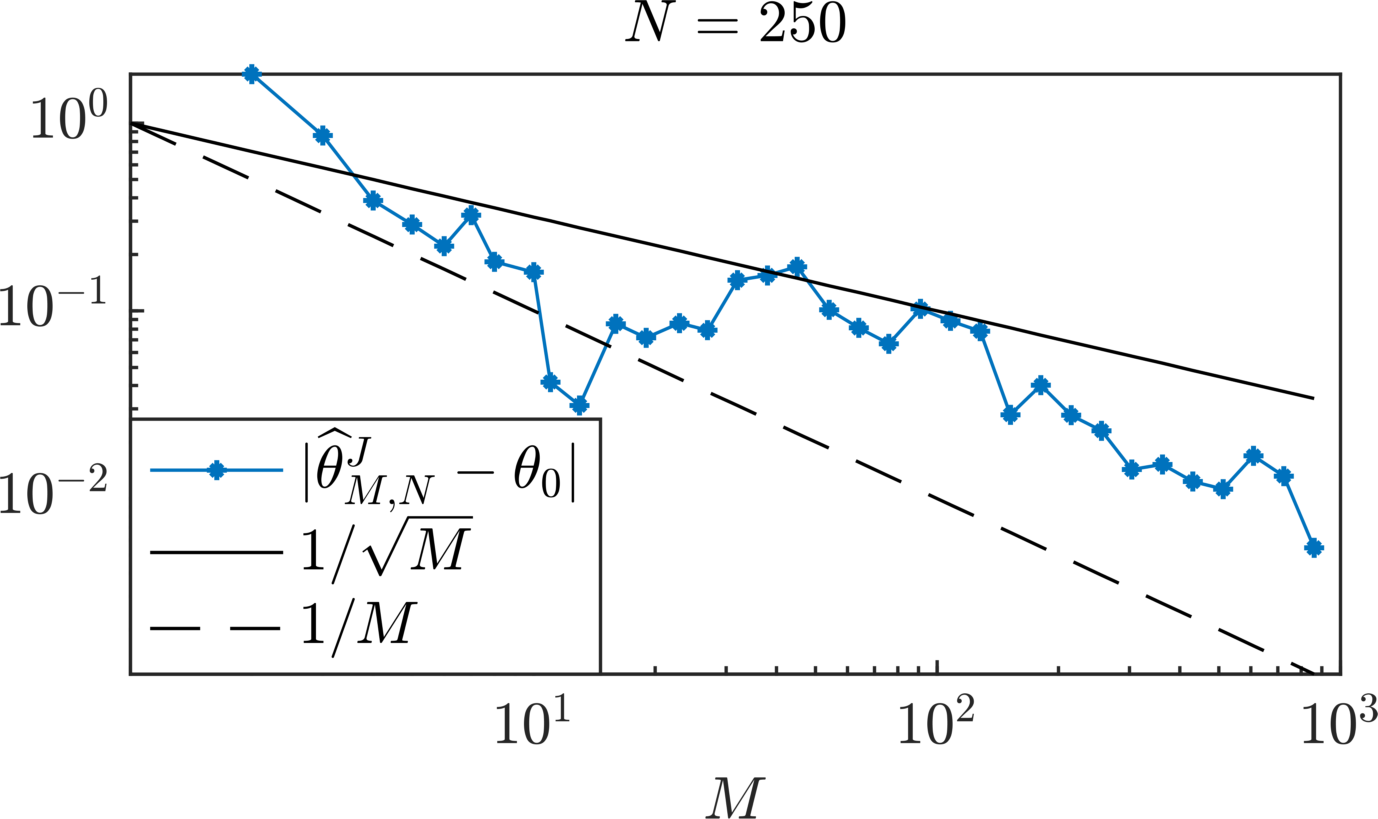} \hspace{0.2cm}
\includegraphics[]{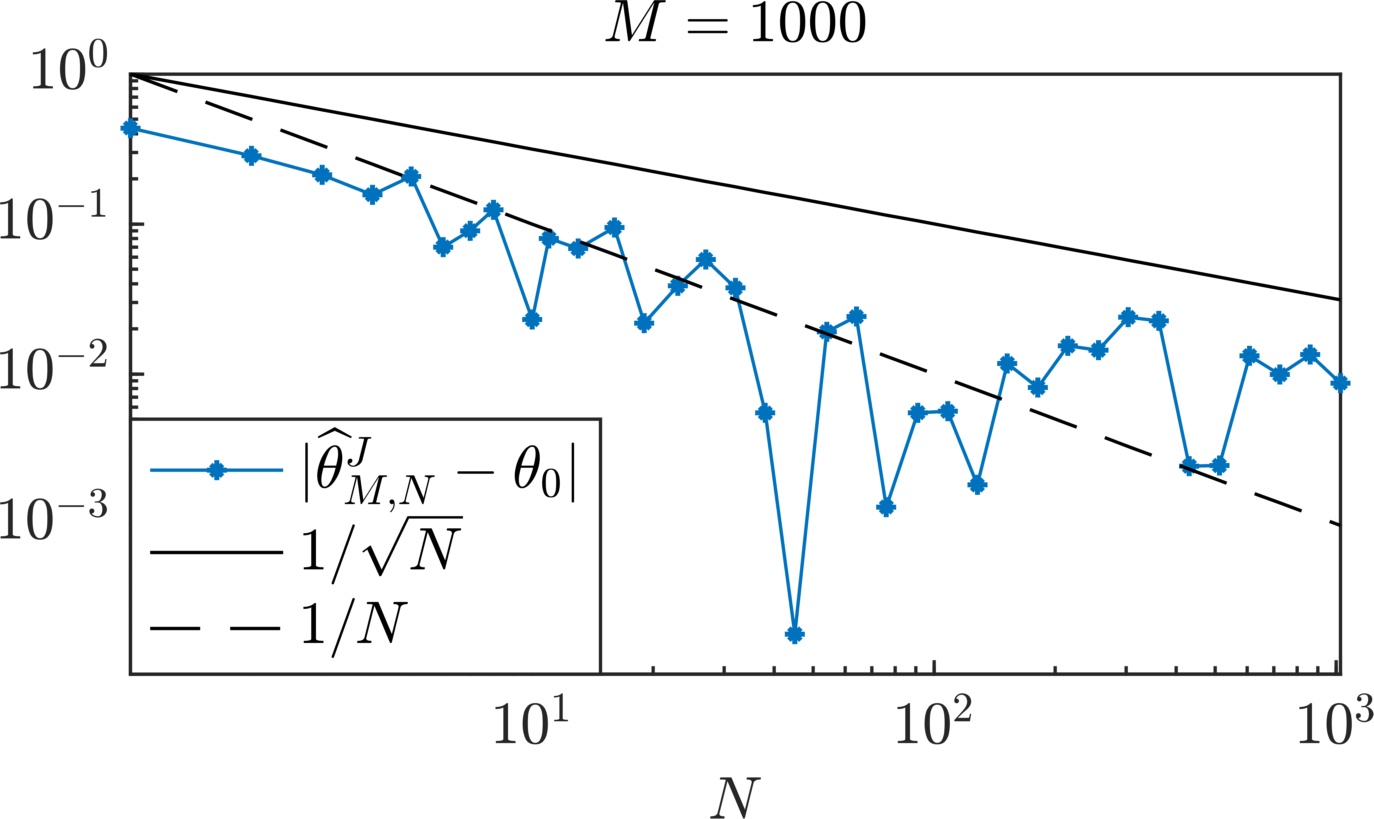}
\caption{Rates of convergence for the Ornstein--Uhlenbeck potential with respect to the number $M$ of observations and $N$ of particles, for the estimator $\widehat \theta^J_{M,N}$ with $J=1$.}
\label{fig:rate_OU}
\end{figure}

We consider the setting of \cref{exa:CW_eigen} choosing $\sigma = 1$, i.e., the interacting particles system reads 
\begin{equation} \label{eq:num_OU}
\dd X_t^{(n)} = - X_t^{(n)} \dd t - \kappa \left( X_t^{(n)} - \bar X_t^N \right) \dd t + \sqrt{2} \dd B_t^{(n)}, \qquad n = 1, \dots, N,
\end{equation}
and we aim to estimate the interaction parameter $\kappa$, so we write $\theta = \kappa$. We set $\kappa = 0.5$ and the number of eigenvalues and eigenfunctions $J = 1$ with $\psi_1(x;\theta) = x$, so that we can employ the analytical expression of our estimator given in \eqref{eq:estimator_OU}. In \cref{fig:sensitivity_OU_MN} we perform a sensitivity analysis for the estimator $\widehat \theta_{M,N}^1$ fixing $\Delta = 1$, varying the number $M$ of observations and $N$ of particles and choosing as other parameter respectively $N = 250$ and $M = 1000$, for which convergence has been reached. The blue line is the estimation given by one single particle while the red line is obtained by averaging the estimations computed employing all the different particles. We notice that convergence is reached when both $N$ and $M$ are large enough and, as expected, the estimation computed by averaging over all the particles stabilizes faster. Moreover, in \cref{fig:sensitivity_OU_J} we fix $M = 1000$ and $N = 250$ and we compare the results for different numbers $J$ of eigenvalues and eigenfunctions employed in the construction of the estimating function. We observe that increasing the value of $J$ does not significantly improves the results, hence it seems preferable to always choose $J = 1$ in order to reduce the computational cost. Finally, in \cref{fig:rate_OU} we verify that the rates of convergence of the estimator $\widehat \theta_{M,N}^1$ towards the exact value $\theta_0$ with respect to the number of observations $M$ and particles $N$ are consistent with the theoretical results given in \cref{thm:main_rate}. In particular, we observe that approximately it holds
\begin{equation}
\abs{\widehat \theta_{M,N}^1 - \theta_0} \simeq \mathcal O \left( \frac{1}{\sqrt M} + \frac{1}{\sqrt N} \right).
\end{equation}

\subsection{Comparison with the maximum likelihood estimator}

\begin{figure}
\centering
\includegraphics[]{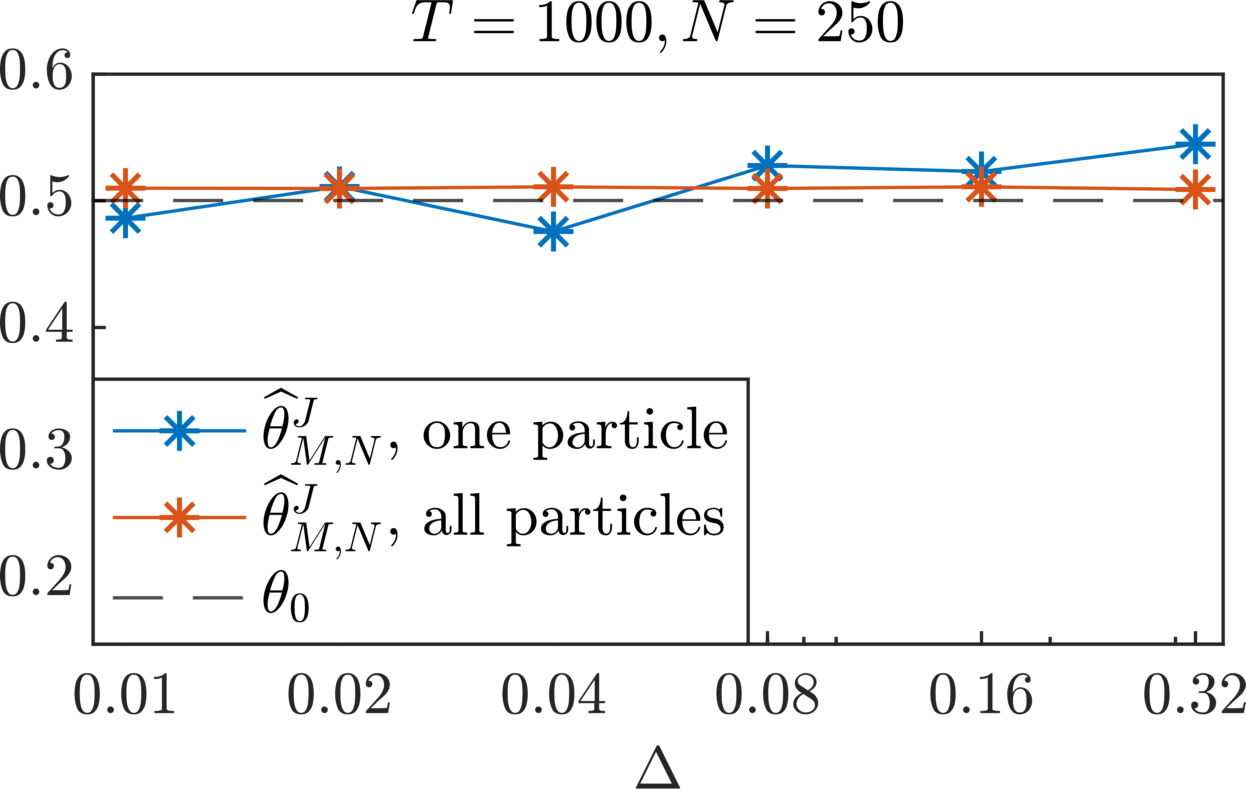} $\quad$ 
\includegraphics[]{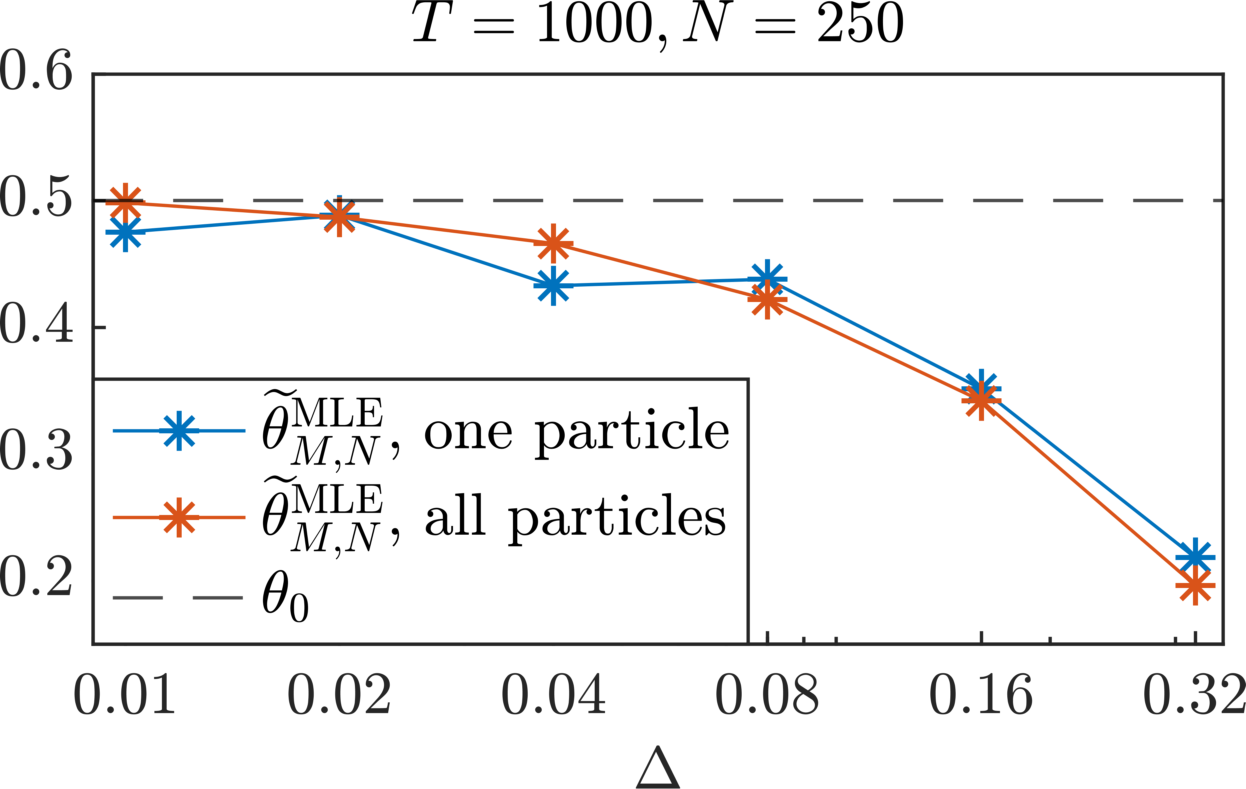}
\caption{Comparison between the estimator $\widehat \theta^J_{M,N}$ with $J=1$ (left) and the maximum likelihood estimator $\widetilde \theta_{M,N}^{\mathrm{MLE}}$ (right) varying the distance $\Delta$ between two consecutive observations for the Ornstein--Uhlenbeck potential.}
\label{fig:comparison_OU_MLE}
\end{figure}

We keep the same setting of \cref{sec:num_sensitivity} and we compare the results of our estimator with a maximum likelihood estimator. In particular, in \cite{Kas90} the MLE for the interacting particles system with continuous observations is rigorously derived. Since for large values of $N$ all the particles are approximately independent and identically distributed and  we are assuming to observe only one particle, we replace the sample mean with the expectation with respect to the invariant measure, i.e., $\bar X_t^N = 0$, and we ignore the sum over all the particles. We then discretize the integrals in the formulation obtaining a modified MLE
\begin{equation} \label{eq:MLE_OU}
\widetilde \theta_{M,N}^{\mathrm{MLE}} = - 1 - \frac{\sum_{m=0}^{M-1} \widetilde X_m^{(n)} (\widetilde X_{m+1}^{(n)} - \widetilde X_m^{(n)})}{\Delta \sum_{m=0}^{M-1} (\widetilde X_m^{(n)})^2}.
\end{equation}
In \cref{fig:comparison_OU_MLE} we fix the final time $T = 1000$ and we repeat the estimation for different values of $\Delta = 0.01 \cdot 2^i$ with $i = 0, \ldots, 5$. We observe that, differently from our estimator, the MLE is unbiased only for small values of the sampling rate $\Delta$, i.e., when the discrete observations approximate well the continuous trajectory. Notice also that, as highlighted by the numerical experiments, our estimator $\widehat \theta_{M,N}^1$ and the MLE $\widetilde \theta_{M,N}^{\mathrm{MLE}}$ defined respectively in \eqref{eq:estimator_OU} and \eqref{eq:MLE_OU} coincide in the limit of vanishing $\Delta$. In fact, we can rewrite equation \eqref{eq:estimator_OU} as
\begin{equation}
\widehat \theta_{M,N}^1 = - 1 - \frac1\Delta \log \left( 1 + \frac{\sum_{m=0}^{M-1} \widetilde X_m^{(n)} (\widetilde X_{m+1}^{(n)} - \widetilde X_m^{(n)})}{\sum_{m=0}^{M-1} (\widetilde X_m^{(n)})^2} \right),
\end{equation}
observe that the fraction in the argument of the logarithm is $\mathcal O (\Delta)$ and employ the asymptotic expansion $\log(1 + x) \sim x$ for $x = o(1)$.

\subsection{Diffusion coefficient} \label{sec:num_diffusion}

\begin{figure}
\centering
\includegraphics[]{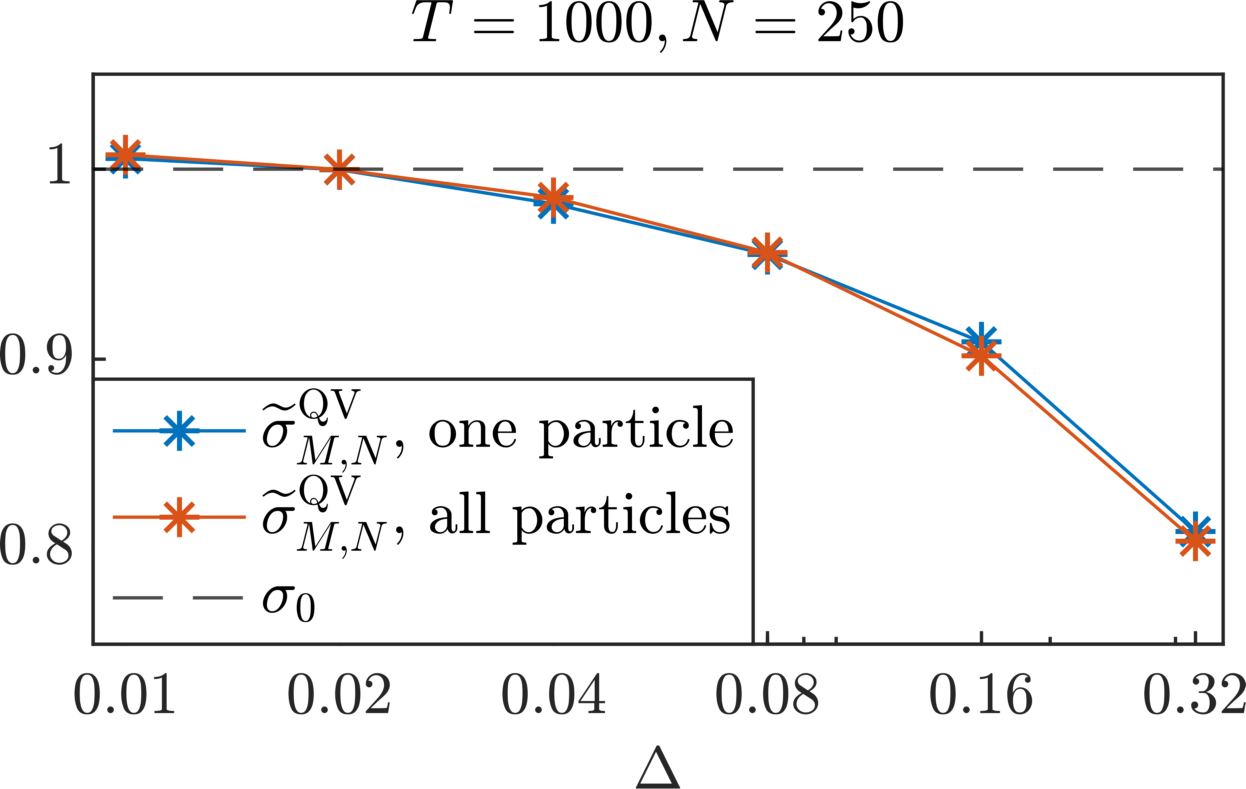}
\caption{Inference of the diffusion coefficient based on the quadratic variation varying the distance $\Delta$ between two consecutive observations for the Ornstein--Uhlenbeck potential.}
\label{fig:diffusion_QV}
\end{figure}

\begin{figure}
\centering
\includegraphics[]{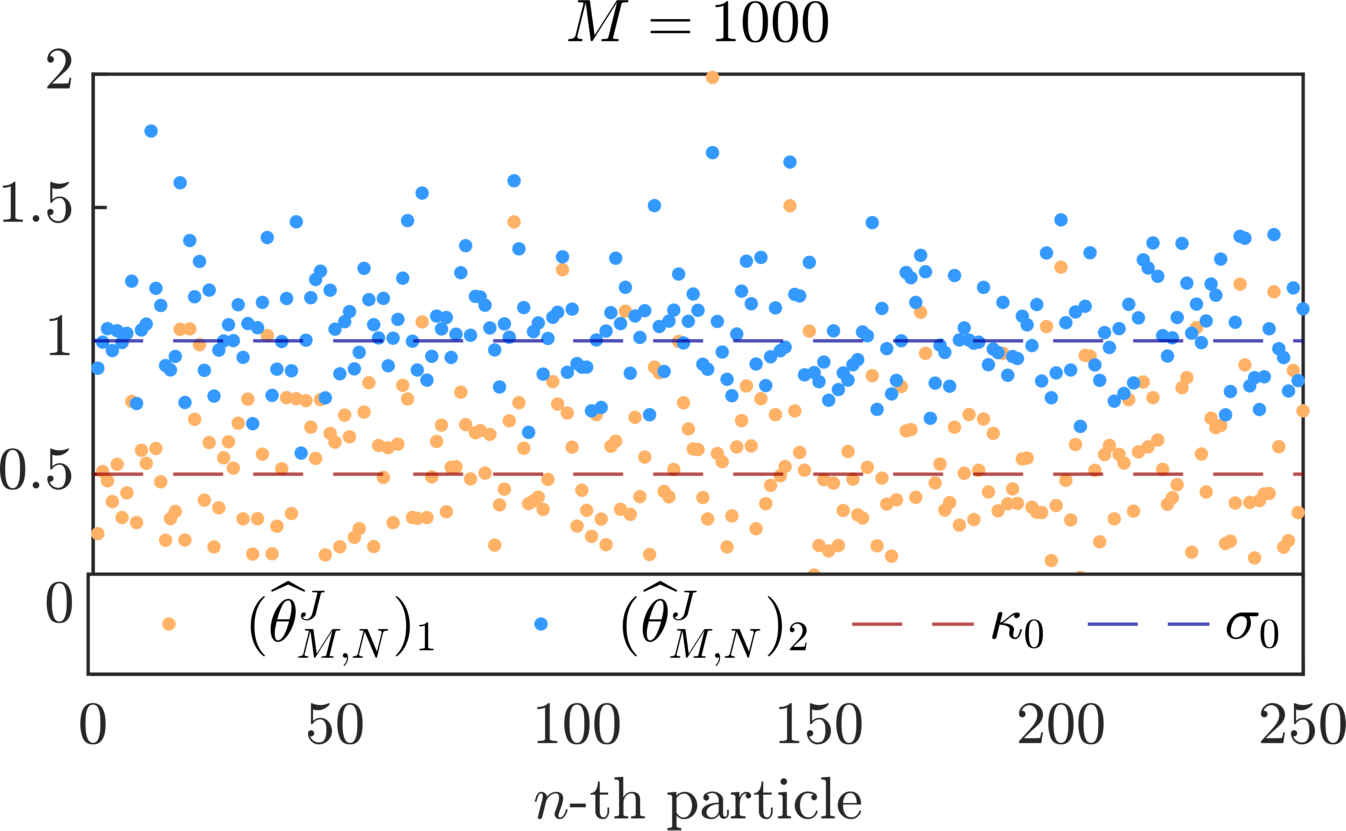} $\quad$ 
\includegraphics[]{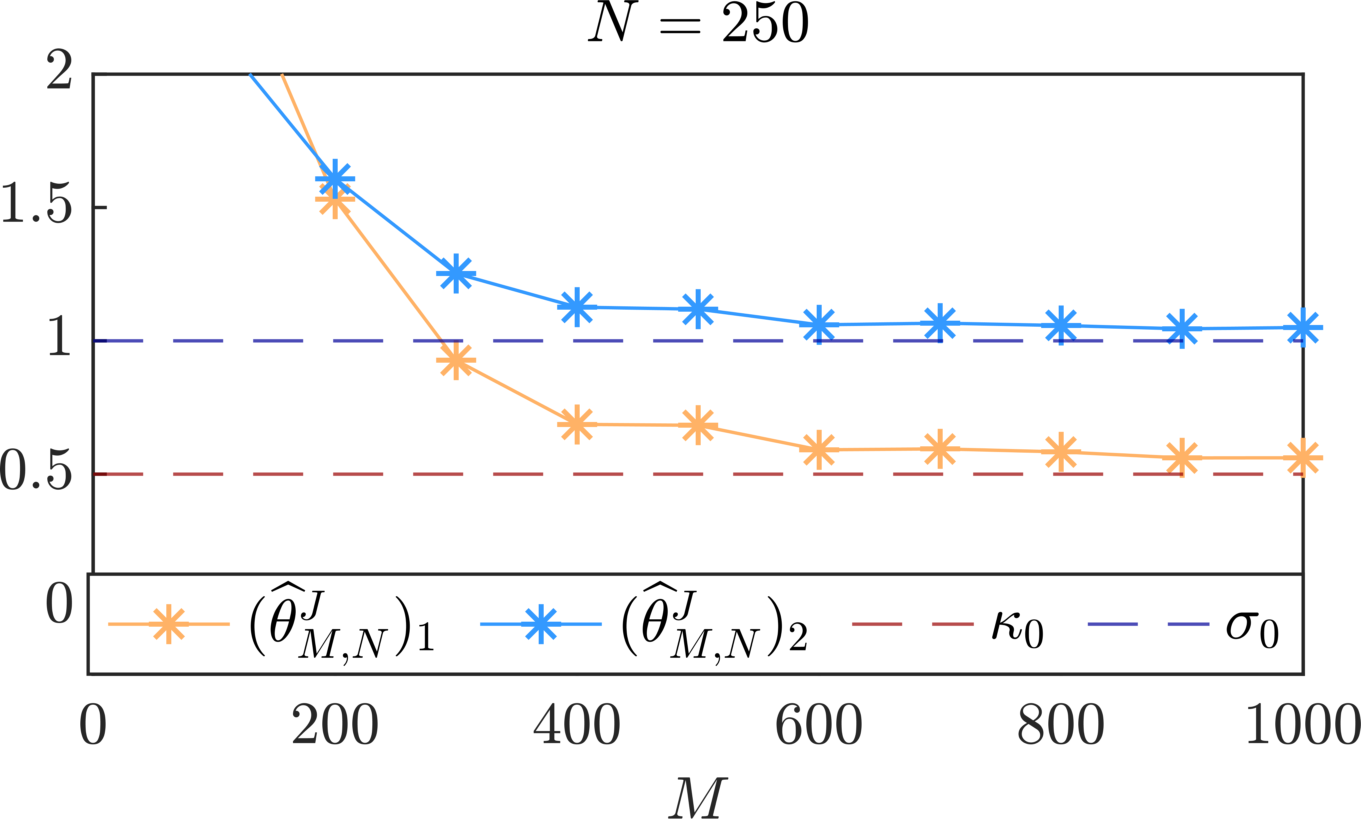}
\caption{Simultaneous inference of the interaction and diffusion coefficients for the Ornstein--Uhlenbeck potential. Left: estimation $\widehat \theta_{M,N}^J$ obtained from each particle with $J = 2$. Right: average of the estimations varying the number of observations.}
\label{fig:diffusion}
\end{figure}

We still consider the setting of \cref{exa:CW_eigen}, but, differently from \cref{sec:num_sensitivity}, we now assume the diffusion coefficient to be unknown and we aim to retrieve the correct values of the interaction parameter and the diffusion coefficient, which are given by $\kappa = 0.5$ and $\sigma = 1$, respectively. We set the number of particles $N = 250$ and the number of observations $M = 1000$. A first approach consists in first estimating the diffusion coefficient alone employing the quadratic variation and then infer the interaction parameter as in the previous numerical experiments. In particular, the diffusion coefficient can be approximated as
\begin{equation}
\widetilde \sigma^{\mathrm{QV}}_{M,N} = \frac{1}{2\Delta M} \sum_{m=0}^{M-1} (\widetilde X_{m+1}^{(n)} - \widetilde X_m^{(n)})^2.
\end{equation}
However, this estimator is asymptotically unbiased only in the limit of $\Delta$ vanishing and is therefore reliable only if the sampling rate is sufficiently small. In fact, one can prove that 
\begin{equation}
\lim_{N \to \infty} \lim_{M \to \infty} \widetilde \sigma^{\mathrm{QV}}_{M,N} = \lim_{N \to \infty} \frac{1}{2\Delta} \E \left[ (X^{(n)}_\Delta - X^{(n)}_0)^2 \right] = \frac{1}{2\Delta} \E \left[ (X_\Delta - X_0)^2 \right],
\end{equation}
which, due to the fact that in the framework of \cref{exa:CW_eigen} $X_t$ at stationarity is a Gaussian process with zero mean and covariance function 
\begin{equation}
\mathcal C(t,s) = \frac{\sigma}{1 + \kappa} e^{-(1 + \kappa)|t - s|},
\end{equation}
implies
\begin{equation}
\lim_{N \to \infty} \lim_{M \to \infty} \widetilde \sigma^{\mathrm{QV}}_{M,N} = \sigma \frac{1 - e^{-(1+\kappa)\Delta}}{(1+\kappa)\Delta},
\end{equation}
where the right-hand side converges to $\sigma$ if $\Delta$ goes to zero. This is also shown in \cref{fig:diffusion_QV} where we estimate the diffusion coefficient for different values of the sampling rate $\Delta = 0.01 \cdot 2^i$ with $i = 0, \ldots, 5$. Hence, if $\Delta$ is far from its vanishing limit we have to follow a different procedure. We now fix $\Delta = 1$ and aim to simultaneously infer the diffusion coefficient and the interaction parameter using our eigenfunction martingale estimators. We then write $\theta = \begin{pmatrix} \kappa & \sigma \end{pmatrix}^\top$ and in order to construct the estimating functions we employ $J = 2$ eigenvalues and eigenfunctions with functions $\psi_1(x;\theta) = \psi_2(x;\theta) = \begin{pmatrix} x^2 & x \end{pmatrix}^\top$. We remark that in the particular case of the Ornstein--Uhlenbeck process it is possible to express the eigenvalues and eigenfunctions analytically and the first two are given by 
\begin{equation} \label{eigen_OU_12}
\begin{aligned}
\lambda_1 &= 1 + \kappa, \hspace{2cm} \phi_1(x;\theta) = x, \\
\lambda_2 &= 2(1 + \kappa), \hspace{1.55cm} \phi_2(x;\theta) = x^2 - \frac{\sigma}{1 + \kappa}.
\end{aligned}
\end{equation}
Note that the first eigenvalue and eigenfunction do not depend on the diffusion coefficient $\sigma$ and therefore they alone do not provide enough information, hence it is important to choose at least $J = 2$. In \cref{fig:diffusion} we show the numerical results. On the left and we plot the estimation computed employing one single particle for all the $N$ particles and we observe that the estimators are concentrated around the exact values. On the other hand, on the right, we average all the estimations previously computed and we plot the results varying the number of observations $M$. We notice that the estimations stabilize fast near the correct coefficients.

\subsection{Central limit theorem}

We keep the same setting of \cref{sec:num_sensitivity} and we validate numerically the central limit theorem which we proved theoretically in \cref{thm:main_normal}. In this particular case, the asymptotic variance $\Gamma_0^J$ can be computed analytically. In fact, the mean field limit of \eqref{eq:num_OU} at stationarity is 
\begin{equation}
\dd X_t = - (1+\kappa) X_t \dd t + \sqrt{2} \dd B_t,
\end{equation}
and its solution $(X_t)_{t \in [0,T]}$ is a Gaussian process, i.e., $X \sim \mathcal{GP}(m(t), \mathcal C(t,s))$, where $m(t) = 0$ and
\begin{equation}
\mathcal C(t,s) = \frac{1}{1 + \kappa} e^{-(1+\kappa)|t-s|}.
\end{equation}
Moreover, we have
\begin{equation}
h_1(x,y;\theta) = \Delta e^{-(1+\kappa)\Delta} x^2 \qquad \text{and} \qquad \ell_{1,1}(x,y;\theta) = x^2 \left( y^2 - e^{-2(1+\kappa)\Delta} x^2 \right),
\end{equation}
and therefore we obtain 
\begin{equation}
\Gamma_0^J = \frac{e^{2(1+\kappa)\Delta} - 1}{\Delta^2}.
\end{equation}
We then fix the number of particles $N = 1500$, the number of observations $M = 1000$ and the sampling rate $\Delta = 1$. In \cref{fig:CLT_OU} we plot the quantity $\sqrt M (\widehat \theta_{M,N}^J - \theta_0)$ for any particle $n = 1, \dots, N$ and for $L = 500$ realizations of the Brownian motion and we observe that it is approximately distributed as $\mathcal N(0, \Gamma_0^J)$ accordingly to the theoretical result.

\begin{figure}
\centering
\includegraphics[]{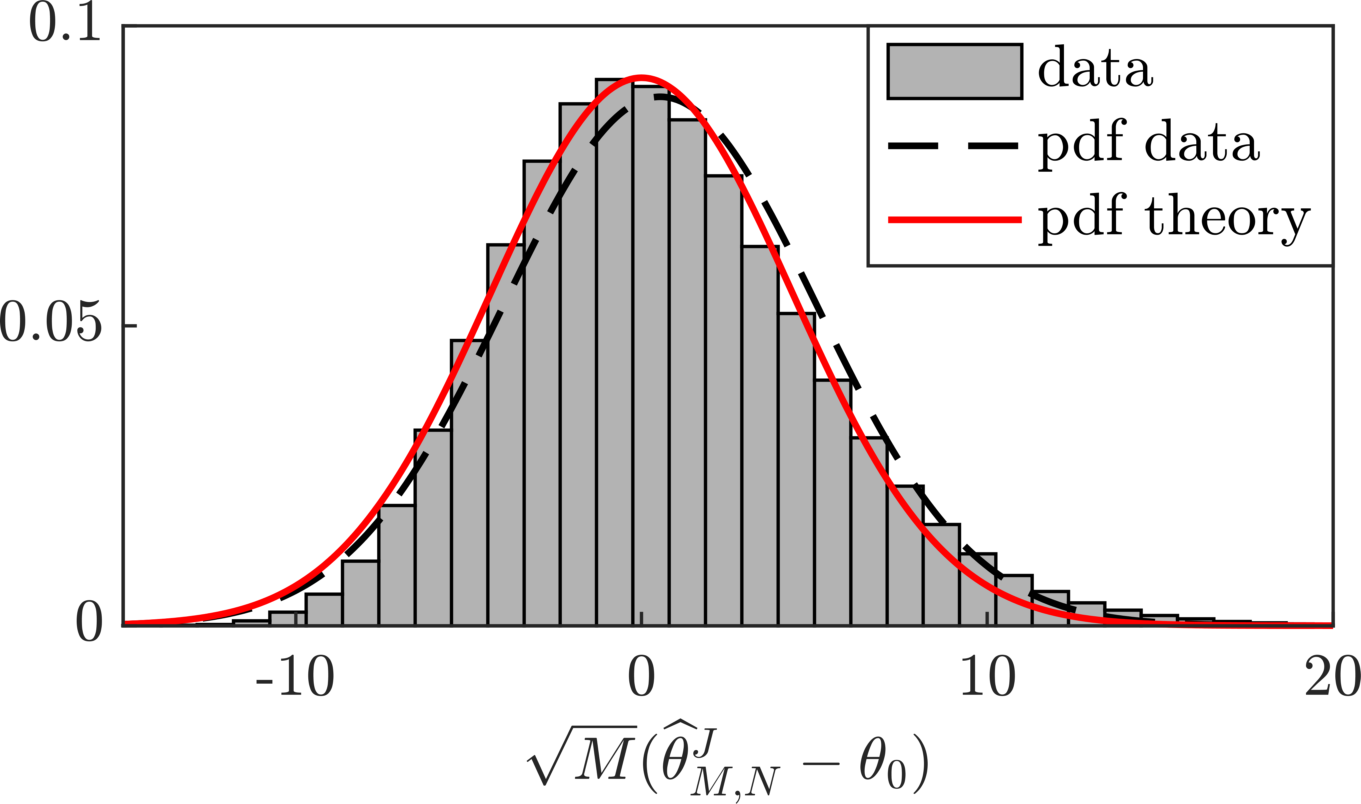}
\caption{Central limit theorems for the Ornstein--Uhlenbeck potential, for the estimator $\widehat \theta^J_{M,N}$ with $J=1$.}
\label{fig:CLT_OU}
\end{figure}

\subsection{Double well potential} \label{sec:bistable}

\begin{figure}
\centering
\includegraphics[]{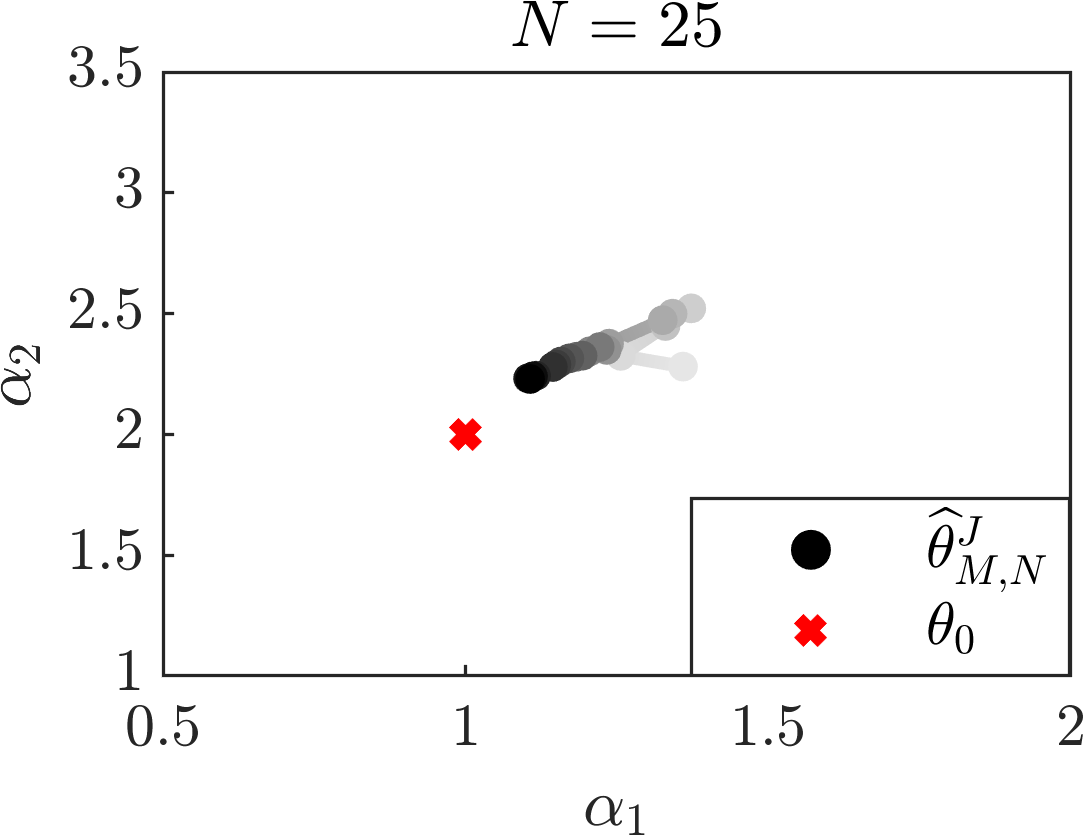}
\includegraphics[]{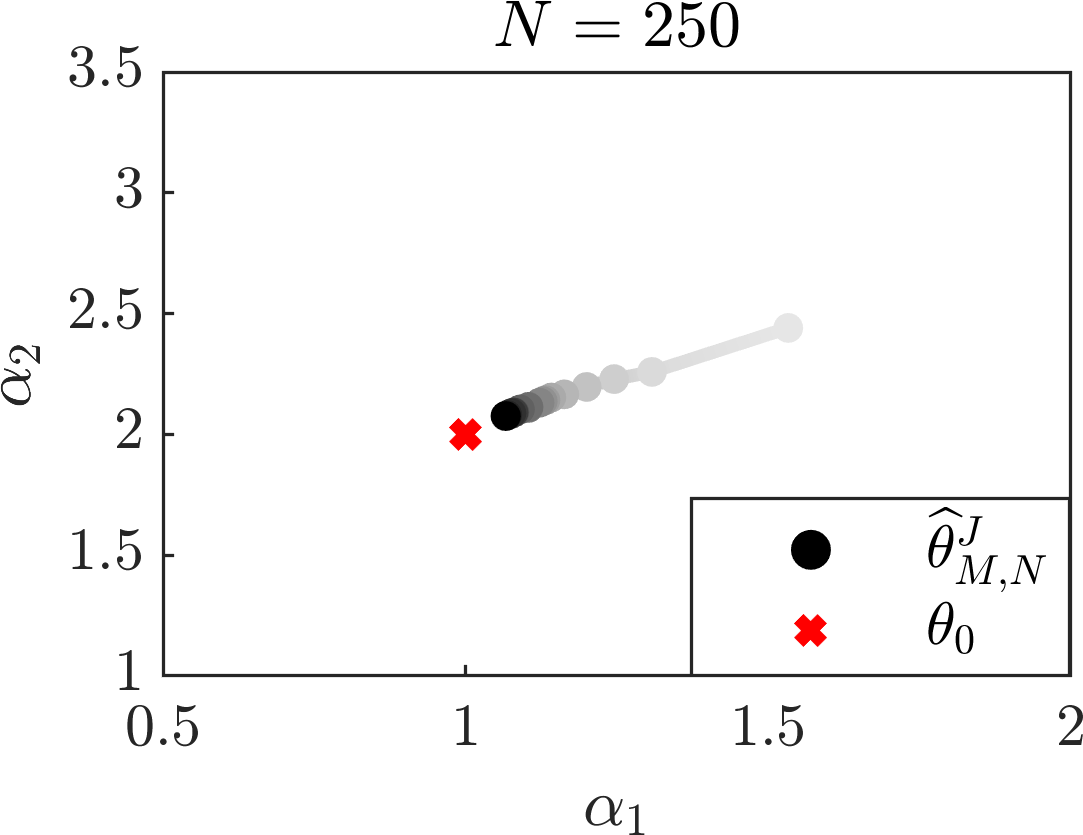} \hspace{0.5cm}
\includegraphics[]{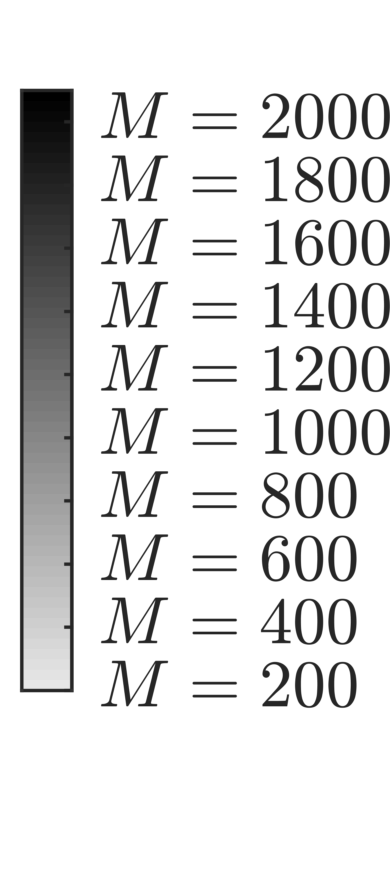} \\
\vspace{0.3cm}
\includegraphics[]{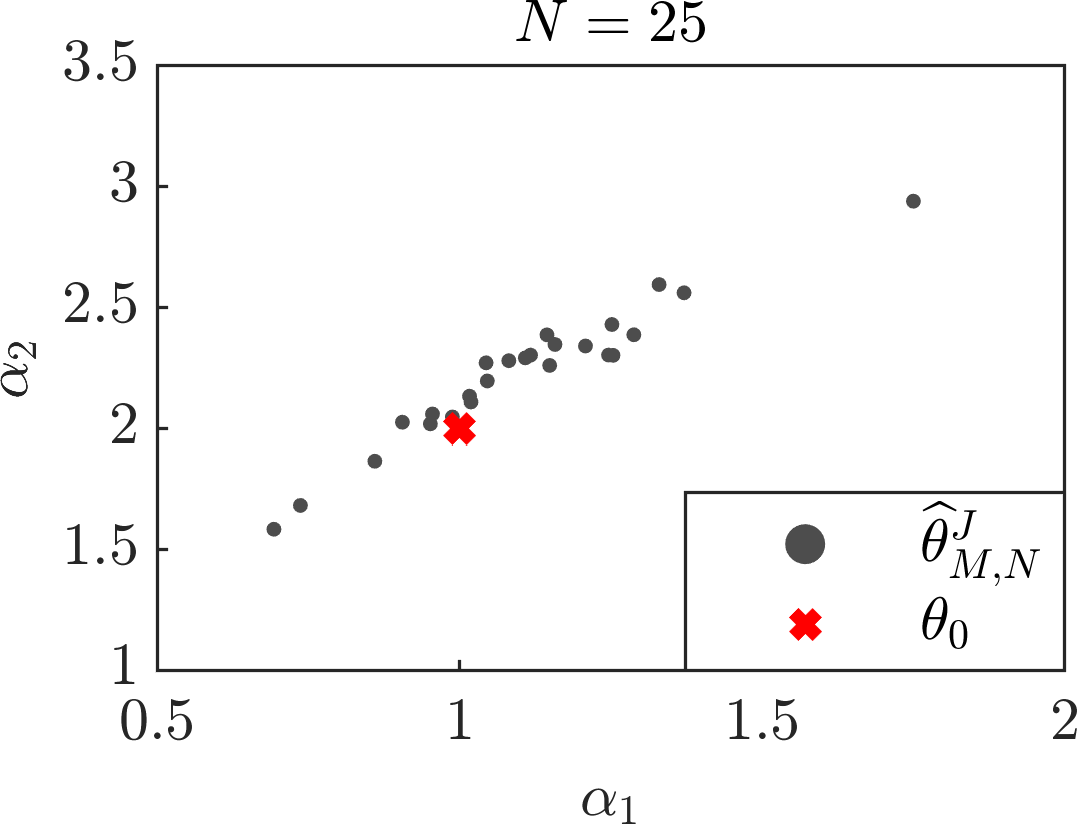} \hspace{0.15cm}
\begin{overpic}[]{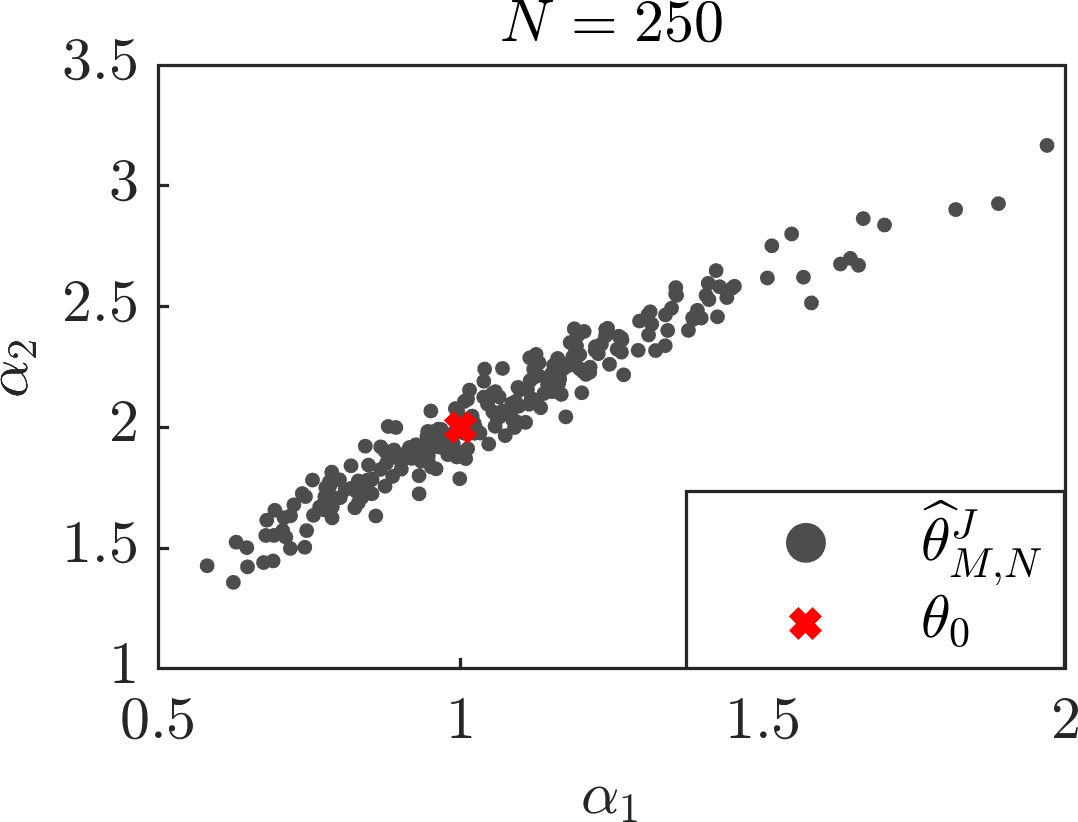} \put(200,60){$M = 2000$} \end{overpic} \hspace{0.6cm}
\phantom{\includegraphics[]{figures/legend_gray}}
\caption{Inference of the two-dimensional drift coefficient of the double well potential below the phase transition. Top: average of the estimations $\widehat \theta_{M,N}^J$ with $J = 1$ varying the number of observations. Bottom: scatter plot of the estimations obtained from each particle.}
\label{fig:bistable}
\end{figure}

\begin{figure}
\centering
\includegraphics[]{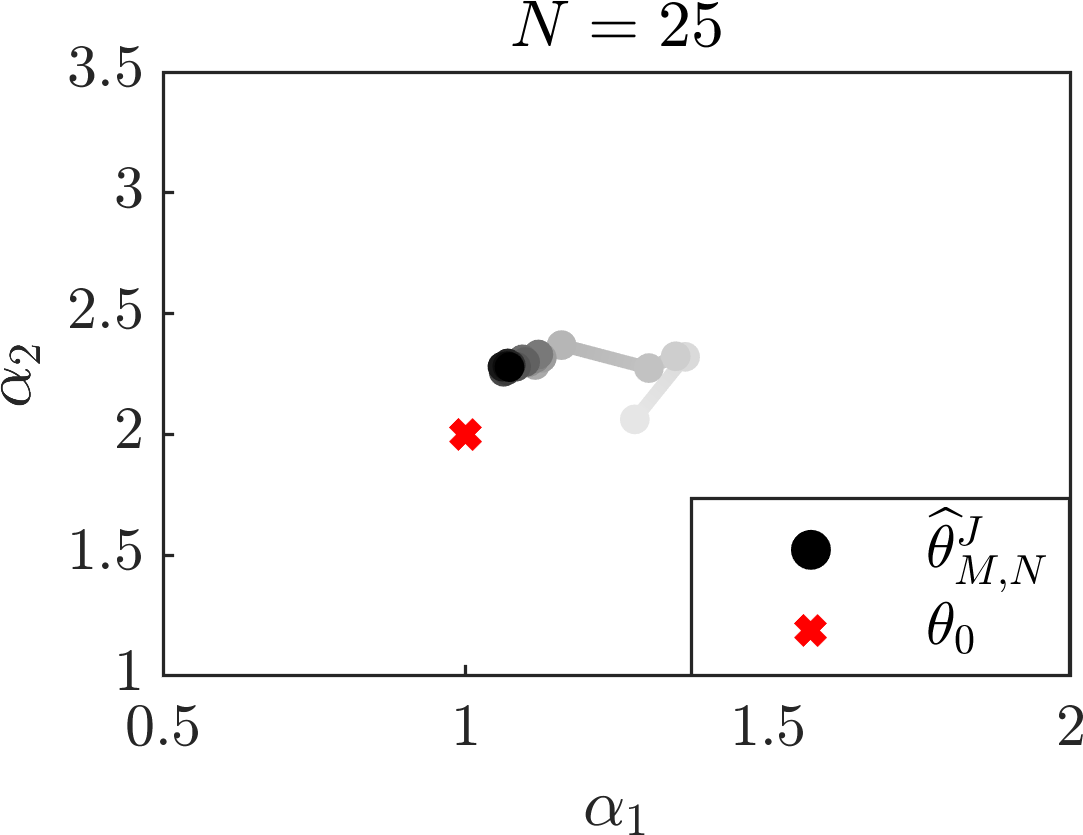}
\includegraphics[]{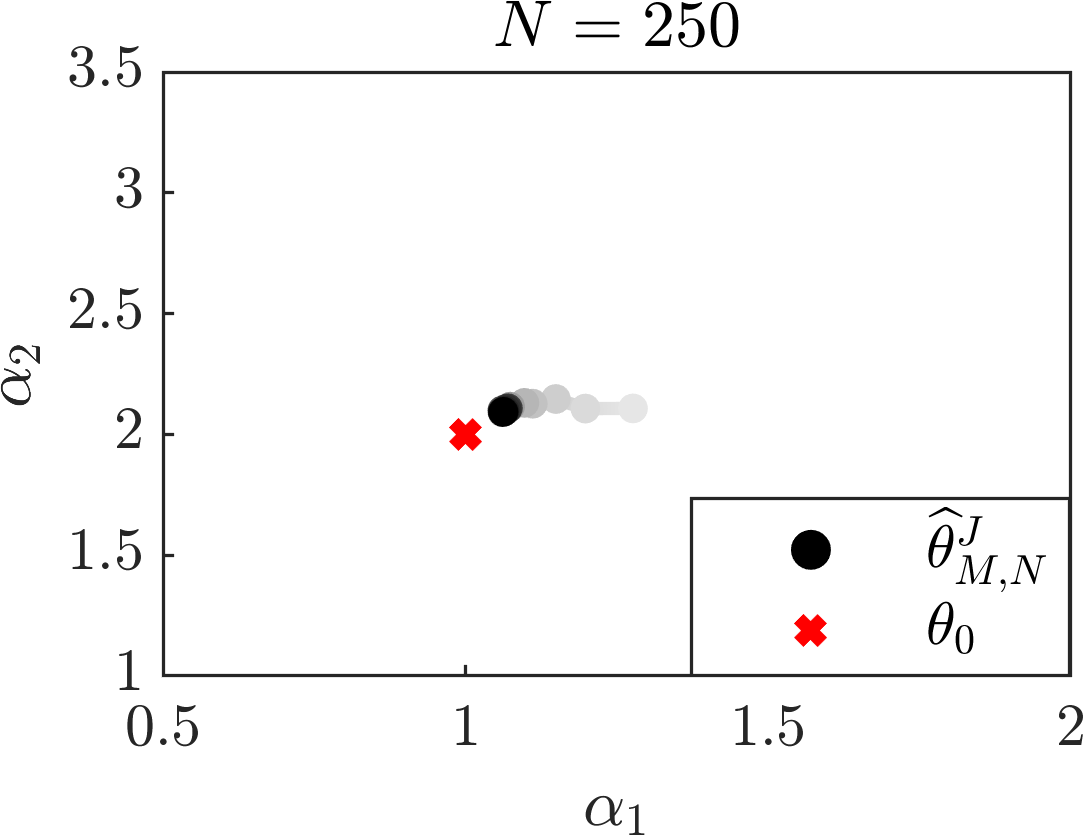} \hspace{0.5cm}
\includegraphics[]{figures/legend_gray} \\
\vspace{0.3cm}
\includegraphics[]{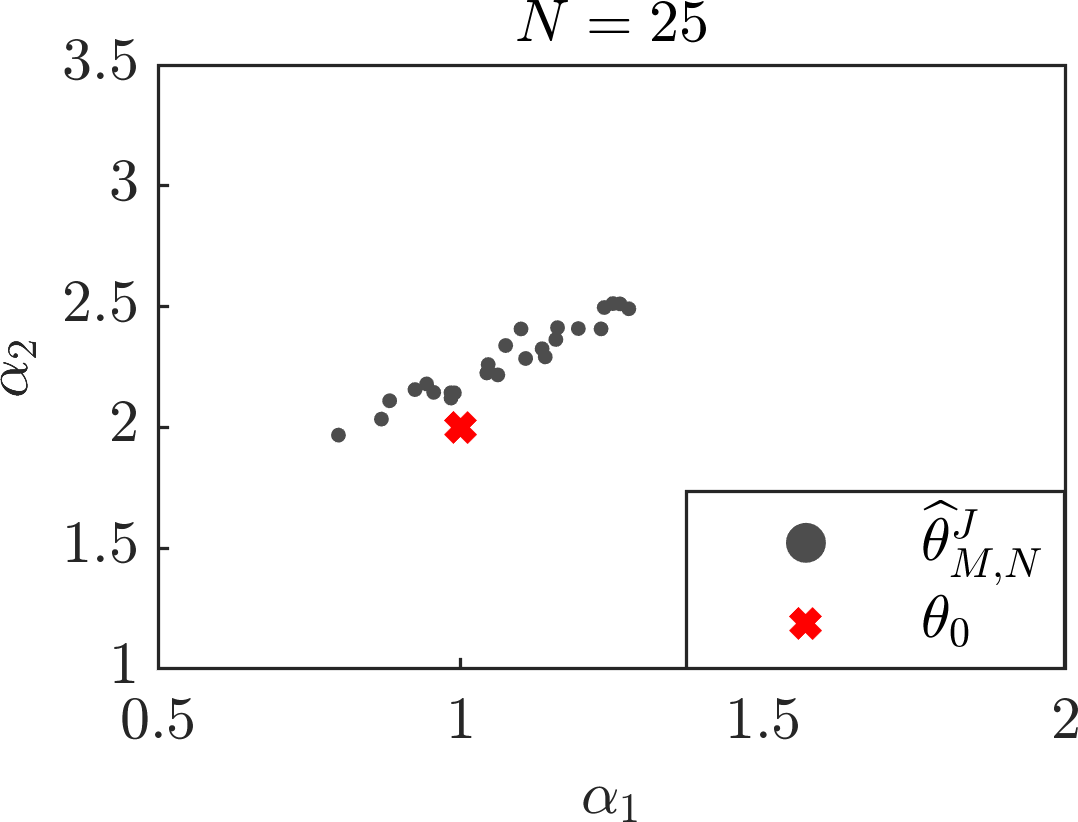} \hspace{0.15cm}
\begin{overpic}[]{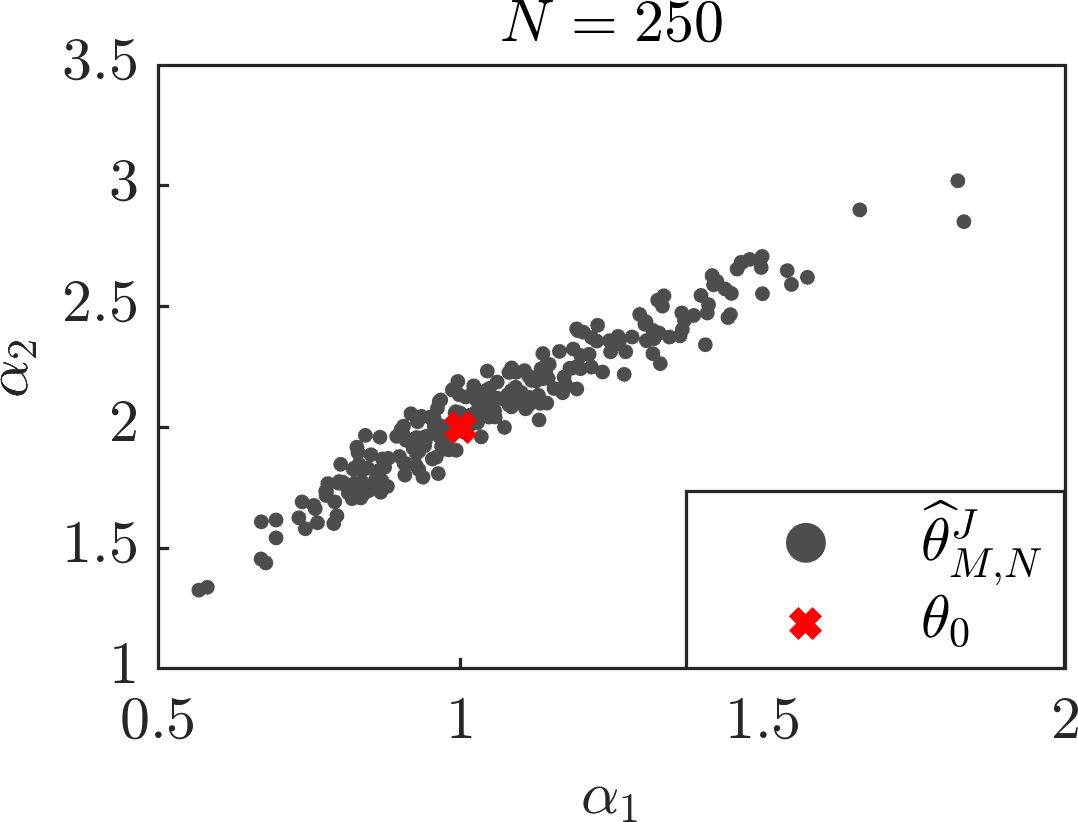} \put(200,60){$M = 2000$} \end{overpic} \hspace{0.6cm}
\phantom{\includegraphics[]{figures/legend_gray}}
\caption{Inference of the two-dimensional drift coefficient of the double well potential above the phase transition. Top: average of the estimations $\widehat \theta_{M,N}^J$ with $J = 1$ varying the number of observations. Bottom: scatter plot of the estimations obtained from each particle.}
\label{fig:bistable_mean_unknown}
\end{figure}

We consider the setting of \cref{exa:CW} and we analyse the double well potential, i.e., we let the confining potential $V(\cdot; \alpha)$ be 
\begin{equation}
V(x; \alpha) = \alpha \cdot \begin{pmatrix} \frac{x^4}{4} & -\frac{x^2}{2} \end{pmatrix}^\top,
\end{equation}
with $\alpha = \begin{pmatrix} 1 & 2 \end{pmatrix}^\top$, which is the parameter that we aim to estimate, so we write $\theta = \alpha$. Moreover, we set the interaction term $\kappa = 0.5$ and the number of observations $M = 2000$ with sampling rate $\Delta = 0.5$. Finally, to construct the estimating functions we use $J = 1$ eigenfunctions and eigenvalues and we employ the function $\psi_1(x;\theta) = \begin{pmatrix} x & x^3 \end{pmatrix}^\top$. We remark that this example does not fit in \cref{ass:potential}, but if the diffusion coefficient $\sigma$ is chosen sufficiently large, then we are below the phase transition and the mean field limit admits a unique invariant measure \cite{Daw83}, so the theory applies. However, when the diffusion coefficient $\sigma$ is below the critical noise strength, then a continuous phase transition occurs and two stationary states exist \cite{GoP18}. In particular, the transition point occurs at $\sigma \simeq 0.6$ with these data. We therefore perform two numerical experiments, one below and one above the phase transition, setting $\sigma = 0.75$ and $\sigma = 0.5$. In the former we have a unique invariant measure, so we can follow the usual approach, while in the latter we do not know in which state the data are converging. Nevertheless, the invariant distribution is known up to the first moment by equation \eqref{eq:rho_CW}, so we first estimate the expectation using the law of large numbers with the available observations and then repeat the same procedure as in the previous case. In \cref{fig:bistable,fig:bistable_mean_unknown} we plot the results of these two experiments. On the top of the figures we plot the evolution of our estimator varying the number of observations $M$ for two different values of the number of particles, in particular $N = 25$ and $N = 250$. We observe that the estimator approaches the correct drift coefficient $\alpha$ as the number of observations $M$ increases and, as expected, the final approximation is better when the number of particles is sufficiently large. Moreover, on the bottom of the same figures we show the scatter plot of the estimations obtained from each particle with $M = 2000$ observations and we can see that they are concentrated around the exact drift coefficient $\alpha$. We finally remark that we do not notice significant differences between two cases, yielding that the initial estimation of the first moment of the invariant measure does not affect the final results and thus that our methodology can be employed even when multiple stationary states exist.

\subsection{Nonsymmetric confining potential} \label{sec:num_nonsymmetric}

\begin{figure}
\centering
\hspace{0.5cm}
\begin{overpic}[]{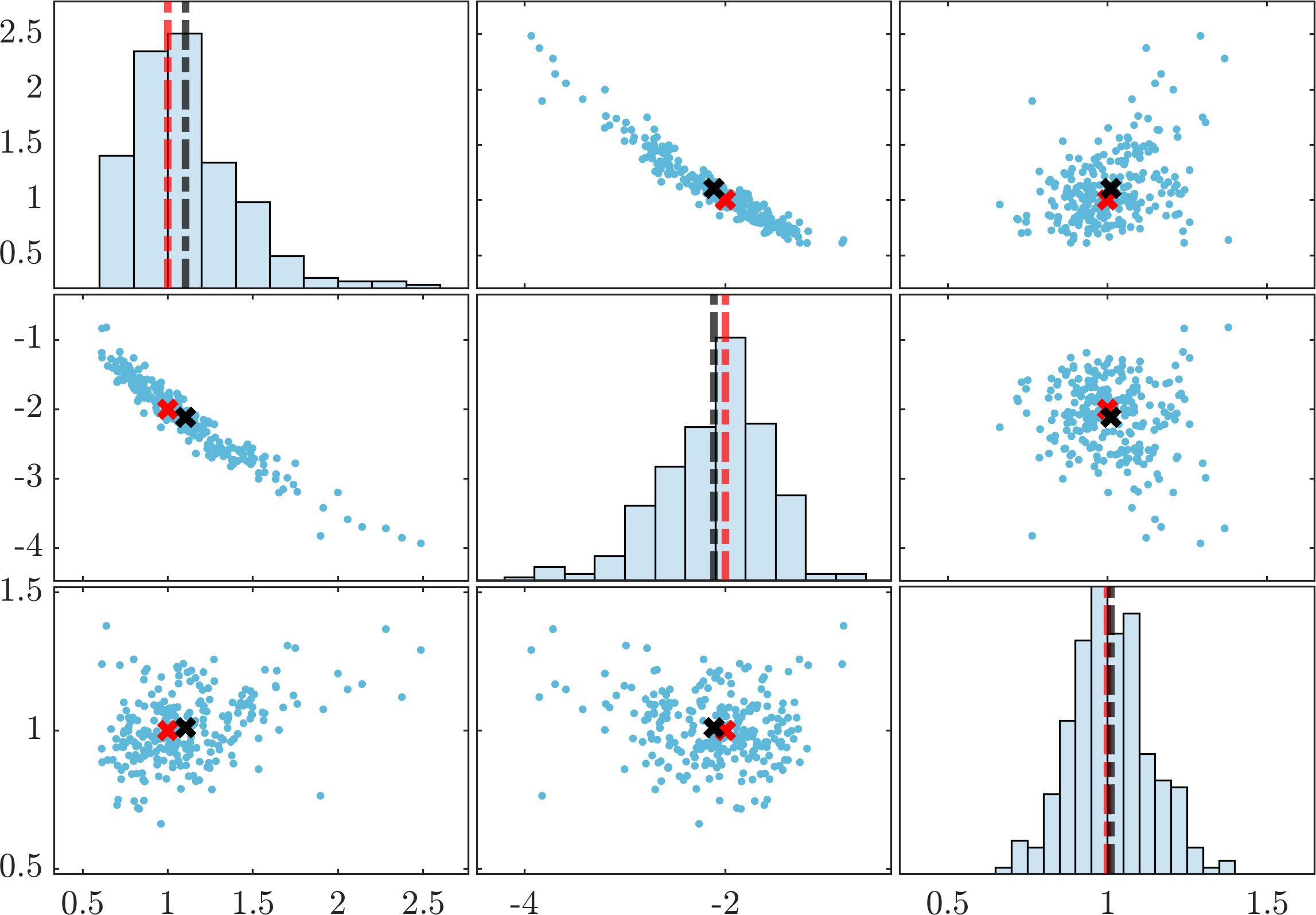} \put(70,-15){$\alpha_1$} \put(180,-15){$\alpha_2$} \put(290,-15){$\alpha_3$} \put(-25,45){$\alpha_3$} \put(-25,120){$\alpha_2$} \put(-25,195){$\alpha_1$} \end{overpic}
\vspace{0.5cm}
\caption{Inference of the three-dimensional drift coefficient of a nonsymmteric potential for the estimator $\widehat \theta_{M,N}^J$ with $J = 1$. Diagonal: histogram of the estimations of each component obtained from all particles. Off-diagonal: scatter plot of the estimations obtained from all particles for two components at a time. Black and red stars/lines represent the average of the estimations and the exact value, respectively.}
\label{fig:tugaut}
\end{figure}

We still consider the same setting of \cref{exa:CW} and we now study the case of a nonsymmetric potential. In particular, we let the confining potential $V(\cdot; \alpha)$ be
\begin{equation}
V(x;\alpha) = \alpha \cdot \begin{pmatrix} \frac{x^4}{4} & \frac{x^2}{2} & x \end{pmatrix}^\top,
\end{equation}
with $\alpha = \begin{pmatrix} 1 & -2 & 1 \end{pmatrix}^\top$, which is the unknown parameter that we want to infer, hence we set $\theta = \alpha$. Notice that the confining potential is given by the sum of the double well potential and a linear term which breaks the symmetry. This type of potentials of the form $V(x) = \sum_{\nu =1}^{\mathscr N} a_{2 \nu} s^{2\nu} + a_1 s$, where $\mathscr N \ge 2$, $a_1, a_2 \in \R$, $a_4, \dots, a_{2(\mathscr N - 1)} \ge 0$ and $a_{2 \mathscr N} > 0$, which is used in the study of metastability and phase transitions and may have arbitrarily deep double wells, has been analyzed in \cite{Tug14,Yos03}. Similarly to the experiment in \cref{sec:bistable}, this example does not satisfy \cref{ass:potential} and more stationary states can exist. In particular, in \cite{Tug14} it has been proved the existence of an invariant measure around each critical point of the potential. We therefore adopt the same strategy as in the second part of \cref{sec:bistable} and, since the invariant measure is known up to the first moment by equation \eqref{eq:rho_CW}, we first approximate the expectation using the sample mean of the available observations, and then proceed with the following steps of the algorithm. We further set the interaction term $\kappa = 0.5$, the diffusion coefficient $\sigma = 1.5$, the number of particles $N = 250$ and the number of observations $M = 2000$ with sampling rate $\Delta = 0.5$. Moreover, to construct the estimating functions we use $J = 1$ eigenfunctions and eigenvalues and we employ the function $\psi_1(x;\theta) = \begin{pmatrix} x & x^2 & x^3 \end{pmatrix}^\top$. In \cref{fig:tugaut} we plot the results of the inference procedure considering two components of the three-dimensional drift coefficient at a time and the single components alone. We observe that the majority of the estimations obtained from all particles are concentrated around the exact values and that their average provides a reliable approximation of the true unknown. A peculiarity of this numerical experiment is the relationship between the first and second components of the estimated drift coefficient, in fact one increases when the other decreases and vice-versa, meaning that the two approximations appear to be correlated.

\section{Proof of the main results} \label{sec:proofs}

In this section we present the proof of \cref{thm:main_unbiased,thm:main_rate,thm:main_normal}, which are the main results of this work. We first recall that due to \cite[Lemma 2.3.1]{Gan08} the solution of the interacting particle system $X_t^{(n)}$ and of its mean field limit $X_t$ have bounded moments of any order, in particular there exists a constant $C > 0$ independent of $N$ such that for all $t \in [0,T]$, $n = 1, \dots, N$ and $q \ge 1$
\begin{equation} \label{eq:bounded_moments}
\E \left[ \abs{X_t^{(n)}}^q \right]^{1/q} \le C \qquad \text{and} \qquad \E \left[ \abs{X_t}^q \right]^{1/q} \le C.
\end{equation}
Moreover, in \cite[Theorem 3.3]{Mal01} it is shown that each particle converges to the solution of the mean field limit with the same Brownian motion in $L^2$, i.e, that
\begin{equation} \label{eq:convergence2meanFieldLimit}
\sup_{t \in [0,T]} \E \left[ \abs{X_t^{(n)} - X_t}^2 \right]^{1/2} \le \frac{C}{\sqrt N},
\end{equation}
where the constant $C$ is also independent of the final time $T$. We also state here a formula which has been proved in \cite{KeS99} and will be crucial in the last part of the proof
\begin{equation} \label{eq:formula_martingale}
\E^{\mu_{\theta_0}} [ \phi_j(X_\Delta; \theta_0) \mid X_0 = x ] = e^{-\lambda_j(\theta_0) \Delta} \phi_j(x;\theta_0), \qquad \text{for all } j = 1, \dots, J,
\end{equation}
where $\theta_0$ is the true parameter which generates the path $(X_t)_{t \in [0,T]}$ and $\E^{\mu_{\theta_0}}$ denotes the fact that $X_0 \sim \mu_{\theta_0}$. Before entering the main part of the proof, we introduce some notation and technical results which will be used later. We finally remark that all the constants will be denoted by $C$ and their value can change from line to line.

\subsection{Limits of the estimating function and its derivative}

Let us first define the following vector-valued functions $\mathbb G_M^J(\theta), \mathcal G_N^J(\theta), \mathscr G^J(\theta) \colon \R^p \to \R^p$ and matrix-valued functions $\mathbb H_M^J(\theta), \mathcal H_N^J(\theta), \mathscr H^J(\theta) \colon \R^p \to \R^{p \times p}$
\begin{equation}
\label{eq:def_Hfunctions}
\begin{aligned}
&\mathbb G_M^J(\theta) \defeq \frac1M \sum_{m=0}^{M-1} \sum_{j=1}^J g_j(\widetilde X_m, \widetilde X_{m+1}; \theta), \quad &&\mathbb H_M^J(\theta) \defeq \frac1M \sum_{m=0}^{M-1} \sum_{j=1}^J h_j(\widetilde X_m, \widetilde X_{m+1}; \theta), \\
&\mathcal G_N^J(\theta) \defeq \sum_{j=1}^J \E^{\mu^N_{\theta_0}} \left[ g_j(X_0^{(n)}, X_\Delta^{(n)}; \theta) \right], && \mathcal H_N^J(\theta) \defeq \sum_{j=1}^J \E^{\mu^N_{\theta_0}} \left[ h_j(X_0^{(n)}, X_\Delta^{(n)}; \theta) \right], \\
&\mathscr G^J(\theta) \defeq \sum_{j=1}^J \E^{\mu_{\theta_0}} \left[ g_j(X_0, X_\Delta; \theta) \right], && \mathscr H^J(\theta) \defeq \sum_{j=1}^J \E^{\mu_{\theta_0}} \left[ h_j(X_0, X_\Delta; \theta) \right].
\end{aligned}
\end{equation}

The following lemma then shows that these quantities are bounded in a suitable norm and thus well defined.

\begin{lemma} \label{lem:boundedG}
Under \cref{ass:potential,ass:functions_psi} there exists a constant $C > 0$ independent of $M,N$ such that for all $q \ge 1$
\begin{equation}
\begin{alignedat}{4}
(i) \; & \E \left[ \norm{G_{M,N}^J(\theta)}^q \right] \le C, \hspace{2cm} &&(ii) \; \E \left[ \norm{\mathbb G_M^J(\theta)}^q \right] \le C, \\
(iii) \; & \norm{\mathcal G_N^J(\theta)} \le C, \hspace{2cm} &&(iv) \; \norm{\mathscr G^J(\theta)} \le C.
\end{alignedat}
\end{equation}
\end{lemma}
\begin{proof}
Since the argument is similar for four cases, we only write the details of $(i)$. Using the triangle inequality we have
\begin{equation}
\E \left[ \norm{G_{M,N}^J(\theta)}^q \right] \le \frac{2^{q-1}}{M} \sum_{m=0}^{M-1} \sum_{j=1}^J \E \left[ \norm{\psi_j(\widetilde X_m^{(n)};\theta)}^q \left( \abs{\phi_j(\widetilde X_{m+1}^{(n)};\theta)}^q + \abs{\phi_j(\widetilde X_m^{(n)};\theta)}^q \right) \right],
\end{equation}
and due to the Cauchy--Schwarz inequality we obtain
\begin{equation}
\begin{aligned}
\E \left[ \norm{G_{M,N}^J(\theta)} \right] &\le \frac{2^{q-1}}{M} \sum_{m=0}^{M-1} \sum_{j=1}^J \E \left[ \norm{\psi_j(\widetilde X_m^{(n)}; \theta)}^{2q} \right]^{1/2} \E \left[ \abs{\phi_j (\widetilde X_{m+1}^{(n)}; \theta)}^{2q} \right]^{1/2} \\
&\quad + \frac{2^{q-1}}{M} \sum_{m=0}^{M-1} \sum_{j=1}^J \E \left[ \norm{\psi_j(\widetilde X_m^{(n)}; \theta)}^{2q} \right]^{1/2} \E \left[ \abs{\phi_j(\widetilde X_m^{(n)}; \theta)}^{2q} \right]^{1/2}.
\end{aligned}
\end{equation}
Finally, bound \eqref{eq:bounded_moments} together with the fact that $\psi_j$ and $\phi_j$ are polynomially bounded for all $j = 1, \dots, J$ by \cref{ass:functions_psi} gives the desired result.
\end{proof}

In the next proposition we study the behaviour of the estimating function $G_{M,N}^J$ as the number of observations $M$ and particles $N$ go to infinity.

\begin{proposition} \label{pro:limitsG}
Under \cref{ass:potential,ass:functions_psi} it holds for all $1 \le q < 2$
\begin{equation}
\begin{alignedat}{4}
(i) \; & \lim_{N \to \infty} G_{M,N}^J(\theta) = \mathbb G_M^J(\theta), \quad &&\text{in } L^q, \hspace{2cm} (ii) \; \lim_{M \to \infty} \mathbb G_M^J(\theta) = \mathscr G^J(\theta), \quad \text{in } L^2, \\
(iii) \; & \lim_{M \to \infty}  G_{M,N}^J(\theta) = \mathcal G_N^J(\theta), \quad &&\text{in } L^2, \hspace{2cm} (iv) \; \lim_{N \to \infty} \mathcal G_N^J(\theta) = \mathscr G^J(\theta).
\end{alignedat}
\end{equation}
Moreover, there exists a constant $C > 0$ independent of $M,N$ and $\theta$ such that
\begin{equation}
(i)' \; \E \left[ \norm{G_{M,N}^J(\theta) - \mathbb G_M^J(\theta)}^q \right]^{1/q} \le \frac{C}{\sqrt N}, \hspace{2cm} (iv)' \; \norm{\mathcal G_N^J(\theta) - \mathscr G^J(\theta)} \le \frac{C}{\sqrt N}.
\end{equation}
\end{proposition}
\begin{proof}
Results $(ii)$ and $(iii)$ are direct consequences of \cite[Lemma 3.1]{BiS95} and of the ergodicity of the processes $(X_t^{(n)})_{t \in [0,T]}$ and $(X_t)_{t \in [0,T]}$ given by \cite[Section 1]{GoP18} and \cite[Theorem 3.16]{Mal01}, respectively. Let us now consider cases $(i)$ and $(i)'$. Using the triangle inequality we have
\begin{equation}
\E \left[ \norm{G_{M,N}^J(\theta) - \mathbb G_M^J(\theta)}^q \right] \le \frac{4^{q-1}}{M} \sum_{m=0}^{M-1} \sum_{j=1}^J \left( Q_{m,j}^{(1)} + Q_{m,j}^{(2)} + Q_{m,j}^{(3)} + Q_{m,j}^{(4)} \right),
\end{equation}
where
\begin{equation}
\begin{aligned}
Q_{m,j}^{(1)} &\defeq \E \left[ \norm{\psi_j(\widetilde X_m^{(n)}; \theta)}^q \abs{\phi_j(\widetilde X_{m+1}^{(n)}; \theta) - \phi_j(\widetilde X_{m+1}; \theta)}^q \right], \\
Q_{m,j}^{(2)} &\defeq \E \left[ \norm{\psi_j(\widetilde X_m^{(n)}; \theta)}^q \abs{\phi_j(\widetilde X_m^{(n)}; \theta) - \phi_j(\widetilde X_m; \theta)}^q \right], \\
Q_{m,j}^{(3)} &\defeq \E \left[ \norm{\psi_j(\widetilde X_m^{(n)}; \theta) - \psi_j(\widetilde X_m; \theta)}^q \abs{\phi_j(\widetilde X_{m+1}; \theta)}^q \right], \\
Q_{m,j}^{(4)} &\defeq \E \left[ \norm{\psi_j(\widetilde X_m^{(n)}; \theta) - \psi_j(\widetilde X_m; \theta)}^q \abs{\phi_j(\widetilde X_m; \theta)}^q \right],
\end{aligned}
\end{equation}
and applying the mean value theorem we obtain
\begin{equation}
\begin{aligned}
Q_{m,j}^{(1)} &\le \E \left[ \norm{\psi_j(\widetilde X_m^{(n)}; \theta)}^q \abs{\int_0^1 \phi_j'(\widetilde X_{m+1} + s(\widetilde X_{m+1}^{(n)} - \widetilde X_{m+1}); \theta) \dd s}^q \abs{\widetilde X_{m+1}^{(n)} - \widetilde X_{m+1}}^q \right], \\
Q_{m,j}^{(2)} &\le \E \left[ \norm{\psi_j(\widetilde X_m^{(n)}; \theta)}^q \abs{\int_0^1 \phi_j'(\widetilde X_m + s(\widetilde X_m^{(n)} - \widetilde X_m); \theta) \dd s}^q \abs{\widetilde X_m^{(n)} - \widetilde X_m}^q \right], \\
Q_{m,j}^{(3)} &\le \E \left[ \norm{\int_0^1 \psi_j'(\widetilde X_m + s(\widetilde X_m^{(n)} - \widetilde X_m); \theta) \dd s}^q \abs{\widetilde X_m^{(n)} - \widetilde X_m}^q \abs{\phi_j(\widetilde X_{m+1}; \theta)}^q \right], \\
Q_{m,j}^{(4)} &\le \E \left[ \norm{\int_0^1 \psi_j'(\widetilde X_m + s(\widetilde X_m^{(n)} - \widetilde X_m); \theta) \dd s}^q \abs{\widetilde X_m^{(n)} - \widetilde X_m}^q \abs{\phi_j(\widetilde X_m; \theta)}^q \right].
\end{aligned}
\end{equation}
Then, employing the Hölder inequality with exponents $4/(2-q), 4/(2-q), 2/q$ and since $\phi_j, \phi_j', \psi_j, \psi_j'$ are polynomially bounded by \cref{ass:functions_psi} and $\widetilde X_m^{(n)}, \widetilde X_m$ have bounded moments of any order by \eqref{eq:bounded_moments} we deduce
\begin{equation}
\E \left[ \norm{G_{M,N}^J(\theta) - \mathbb G_M^J(\theta)}^q \right] \le \frac{C}{M} \sum_{m=0}^{M-1} \sum_{j=1}^J \left( \E \left[ (\widetilde X_m^{(n)} - \widetilde X_m)^2 \right]^\frac{q}{2} + \E \left[ (\widetilde X_{m+1}^{(n)} - \widetilde X_{m+1})^2 \right]^\frac{q}{2} \right),
\end{equation}
which due to \eqref{eq:convergence2meanFieldLimit} proves $(i)'$, which directly implies $(i)$. Finally, the proofs of results $(iv)$ and $(iv)'$ are similar to cases $(i)$ and $(i)'$, respectively, and are omitted here.
\end{proof}

\begin{corollary}
Under \cref{ass:potential,ass:functions_psi} it holds for all $1 \le q < 2$
\begin{equation}
\lim_{M,N \to \infty} G_{M,N}^J(\theta) = \mathscr G^J(\theta), \quad \text{in $L^q$}.
\end{equation}
\end{corollary}
\begin{proof}
Employing the triangle inequality we have
\begin{equation}
\E \left[ \norm{G_{M,N}^J(\theta) - \mathscr G^J(\theta)}^q \right] \le 2^{q-1} \left( \E \left[ \norm{G_{M,N}^J(\theta) - \mathbb G_M^J(\theta)}^q \right] + \E \left[ \norm{\mathbb G_M^J(\theta) - \mathscr G^J(\theta)}^q \right] \right),
\end{equation}
where the right-hand side vanishes by $(i)'$ and $(ii)$ in \cref{pro:limitsG}, yielding the desires result.
\end{proof}

The limits considered in \cref{pro:limitsG} are summarized schematically in the following diagram
\begin{center}
\begin{tikzpicture}
\node at (-5,0) (W) {$G_{M,N}^J(\theta)$};
\node at (0,1) (N) {$\mathbb G_M^J(\theta)$};
\node at (0,-1) (S) {$\mathcal G_N^J(\theta)$};
\node at (5,0) (E) {$\mathscr G^J(\theta)$};
\draw [->] (W) -- (N);
\node at (-2.5,0.75) {\emph{in $L^q$}};
\node at (-2,0.25) {$N \to \infty$};
\draw [->] (N) -- (E);
\node at (2.5,0.75) {\emph{in $L^2$}};
\node at (2,0.25) {$M \to \infty$};
\draw [->] (W) -- (S);
\node at (-2.5,-0.75) {\emph{in $L^2$}};
\node at (-2,-0.25) {$M \to \infty$};
\draw [->] (S) -- (E);
\node at (2,-0.25) {$N \to \infty$};
\end{tikzpicture}
\end{center}
where $q \in [1,2)$.

\begin{remark} \label{rem:convergence_derivatives}
Notice that all the results in this section hold true also for the derivatives $\mathbb H_M^J(\theta)$, $\mathcal H_N^J(\theta)$, $\mathscr H^J(\theta)$ with respect to the parameter $\theta$ defined in \eqref{eq:def_Hfunctions}. Since the arguments are analogous we omit the details here.
\end{remark}

\subsection{Zeros of the limits of the estimating function}

The goal of this section is to show that the limits of the estimating functions previously defined admit zeros and to study their asymptotic limit. We already know by \eqref{eq:formula_martingale} that $\mathscr G^J(\theta_0) = 0$, where $\theta_0$ is the true parameter. Then, in the following lemma we consider the zero of the function $\mathcal G_N^J(\theta)$ and its limit as $N \to \infty$.

\begin{lemma} \label{lem:NtoInf}
Under \cref{ass:potential,ass:functions_psi} and if $\det(\mathscr H^J(\theta_0)) \neq 0$ there exists $N_0 > 0$ such that for all $N > N_0$ there exists $\vartheta_N^J \in \Theta$ which solves the system $\mathcal G_N^J(\theta) = 0$ and satisfies $\det (\mathcal H_N^J(\vartheta_N^J)) \neq 0$. Moreover, there exists a constant $C > 0$ independent of $N$ such that
\begin{equation} \label{eq:difference_thetaN}
\norm{\vartheta_N^J - \theta_0} \le \frac{C}{\sqrt N}.
\end{equation}
\end{lemma}
\begin{proof}
We first remark that by \eqref{eq:formula_martingale} we have $\mathscr G^J(\theta_0) = 0$ and, without loss of generality, we can assume that $\det (\mathscr H^J(\theta_0)) > 0$. Let $\delta > 0$ sufficiently small, by point $(iv)'$ in \cref{pro:limitsG} and \cref{rem:convergence_derivatives} we know that $\mathcal H_N^J(\theta)$ converges to  $\mathscr H^J(\theta)$ uniformly in $\theta$ and therefore there exist $N_1 > 0$ and $\epl > 0$ such that for all $N > N_1$ and for all $\theta \in B_\epl(\theta_0)$
\begin{align}
0 &< \det (\mathscr H^J(\theta_0)) - \delta \le \det (\mathcal H_N^J(\theta)) \le \det (\mathscr H^J(\theta_0)) + \delta, \label{eq:determinant_positive} \\
0 &< \norm{\mathscr H^J(\theta_0)^{-1}} - \delta \le \norm{\mathcal H_N^J(\theta)^{-1}} \le \norm{\mathscr H^J(\theta_0)^{-1}} + \delta. \label{eq:norm_inverse_positive}
\end{align}
Hence, due to equation \eqref{eq:determinant_positive} and applying the inverse function theorem we deduce the existence of $\eta > 0$ such that
\begin{equation}
B_\eta(\mathcal G_N^J(\theta_0)) \subseteq \mathcal G_N^J(B_\epl(\theta_0)).
\end{equation}
Notice that the radius $\eta > 0$ can be chosen independently of $N > N_1$. In fact, by the proof of \cite[Theorem 2.3]{PaS95} and \cite[Lemma 1.3]{Lan93} we observe that $\eta$ is dependent on the radius $\epl$ of the ball $B_\epl(\theta_0)$ and the quantity $\norm{\mathcal H_N^J(\theta_0)^{-1}}$, which can be bounded independently of $N > N_1$ due to estimate \eqref{eq:norm_inverse_positive}. Moreover, since
\begin{equation}
\lim_{N \to \infty} \mathcal G_N^J(\theta_0) = \mathscr G^J(\theta_0) = 0,
\end{equation}
then there exists $N_2 > 0$ such that for all $N > N_2$ we have $0 \in B_\eta(\mathcal G_N^J(\theta_0))$. Therefore, setting $N_0 = \max \{ N_1, N_2 \}$ for all $N > N_0$ there exists $\vartheta_N^J \in B_\epl(\theta_0)$ such that $\mathcal G_N^J(\vartheta_N^J) = 0$, which proves the existence. Furthermore, equation \eqref{eq:determinant_positive} gives $\det (\mathcal H_N^J(\vartheta_N^J)) \neq 0$. It now remains to show estimate \eqref{eq:difference_thetaN}. Since the set $\overline{B_\epl(\theta_0)}$ is compact, there exist $\widetilde \vartheta^J \in \overline{B_\epl(\theta_0)}$ and a subsequence $\vartheta_{N_k}^J$ such that 
\begin{equation}
\lim_{k \to \infty} \vartheta_{N_k}^J = \widetilde \vartheta^J.
\end{equation}
By point $(iv)'$ in \cref{pro:limitsG} the function $\mathcal G_N^J(\theta)$ converges to $\mathscr G^J(\theta)$ uniformly in $\theta$, thus we have
\begin{equation}
0 = \lim_{k \to \infty} \mathcal G_{N_k}^J(\vartheta_{N_k}^J) = \lim_{k \to \infty} \left[ \mathcal G_{N_k}^J(\vartheta_{N_k}^J) - \mathscr G^J(\vartheta_{N_k}^J) + \mathscr G^J(\vartheta_{N_k}^J) \right] = \mathscr G^J(\widetilde \vartheta^J),
\end{equation}
which yields $\widetilde \vartheta^J = \theta_0$. This is guaranteed by the fact that $\epl$ can be previously chosen sufficiently small such that $\theta_0$ is the only zero of the function $\mathscr G^J(\theta)$ in $B_\epl(\theta_0)$. Since $\theta_0$ is the unique limit point for the subsequence $\vartheta_{N_k}^J$, it follows that the whole sequence converges. Then, applying the mean value theorem we obtain
\begin{equation}
\mathscr G^J(\vartheta^J_N) - \mathcal G^J_N(\vartheta^J_N) = \mathscr G^J(\vartheta^J_N) - \mathscr G^J(\theta_0) = \left( \int_0^1 \mathscr H^J(\theta_0 + t(\vartheta_N^J - \theta_0)) \dd t \right) (\vartheta_N^J - \theta_0),
\end{equation}
which implies
\begin{equation}
\norm{\vartheta_N^J - \theta_0} \le \norm{\left( \int_0^1 \mathscr H^J(\theta_0 + t(\vartheta_N^J - \theta_0)) \dd t \right)^{-1}} \norm{\mathscr G^J(\vartheta^J_N) - \mathcal G^J_N(\vartheta^J_N)}.
\end{equation}
Since $\vartheta_N^J$ converges to $\theta_0$ as $N$ goes to infinity, then
\begin{equation}
\lim_{N \to \infty} \norm{\left( \int_0^1 \mathscr H^J(\theta_0 + t(\vartheta_N^J - \theta_0)) \dd t \right)^{-1}} = \norm{\mathscr H^J(\theta_0)^{-1}},
\end{equation}
where the right-hand side is well defined because $\det (\mathscr H^J(\theta_0)) \neq 0$. Therefore, if $N$ is sufficiently large there exists a constant $C > 0$ independent of $N$ such that
\begin{equation}
\norm{\left( \int_0^1 \mathscr H^J(\theta_0 + t(\vartheta_N^J - \theta_0)) \dd t \right)^{-1}} \le C,
\end{equation}
which together with point $(iv)'$ in \cref{pro:limitsG} yields estimate \eqref{eq:difference_thetaN} and concludes the proof.
\end{proof}

In the next lemma we study the zero of the random function $\mathbb G_M^J(\theta)$ and its limit as $M \to \infty$. This result is almost the same as \cite[Theorem 4.3]{KeS99}.

\begin{lemma} \label{lem:MtoInf}
Let the assumptions of \cref{lem:NtoInf} hold. Then, an estimator $\widehat \vartheta_M^J$, which solves the equation $\mathbb G_M^J(\theta) = 0$ and is such that $\det (\mathbb H_M^J(\widehat \vartheta_M^J)) \neq 0$, exists with a probability tending to one as $M \to \infty$. Moreover, 
\begin{equation}
\lim_{M \to \infty} \widehat \vartheta_M^J = \theta_0, \qquad \text{in probability},
\end{equation}
and 
\begin{equation}
\lim_{M \to \infty} \sqrt{M} \left( \widehat \vartheta_M^J - \theta_0 \right) = \Lambda^J \sim \mathcal N(\mathbf 0, \Gamma_0^J), \qquad \text{in distribution},
\end{equation}
where $\Gamma_0^J$ is defined in \eqref{eq:variance_normal}.
\end{lemma}
\begin{proof}
The existence of the estimator $\widehat \vartheta_M^J$ which solves the equation $\mathbb G_M^J(\theta) = 0$ with a probability tending to one as $M \to \infty$ and its asymptotic unbiasedness and normality is given by \cite[Theorem 4.3]{KeS99}, whose prove can be found in \cite[Theorem 3.2]{BiS95} and is based on \cite[Theorem A.1]{BaS94}. Moreover, by the last line of the proof of \cite[Theorem 3.2]{BiS95} or by (A.5) in \cite[Theorem 4.3]{KeS99} we have
\begin{equation} \label{eq:convergence_Jacobian}
\lim_{M \to \infty} \mathbb H_M^J(\widehat \vartheta_M^J) = \mathscr H^J(\theta_0), \qquad \text{in probability}, 
\end{equation}
where $\det(\mathscr H^J(\theta_0)) \neq 0$ by assumption. Hence, there exists $\delta > 0$ such that if
\begin{equation}
\norm{\mathbb H_M^J(\widehat \vartheta_M^J) - \mathscr H^J(\theta_0)} \le \delta,
\end{equation}
then $\det(\mathbb H_M^J(\widehat \vartheta_M^J))) \neq 0$. Moreover, for $M$ large enough it holds
\begin{equation}
\Pr \left( \norm{\mathbb H_M^J(\widehat \vartheta_M^J) - \mathscr H^J(\theta_0)} \le \delta \right) \ge 1 - \epl_M,
\end{equation}
where $\epl_M \to 0$ as $M \to \infty$. Let us now define the events 
\begin{equation}
A_M \defeq \left\{ \exists \; \widehat \vartheta_M^J \colon \mathbb G_M^J (\widehat \vartheta_M^J) \right\} \qquad \text{and} \qquad  B_M \defeq \left\{ \norm{\mathbb H_M^J(\widehat \vartheta_M^J) - \mathscr H^J(\theta_0)} \le \delta \right\},
\end{equation}
and notice that by the first part of the proof we have $\Pr(A_M) = p_M$ where $p_M \to 1$ as $M \to \infty$. Then, using the basic properties of probability measures we obtain
\begin{equation}
\begin{aligned}
\Pr \left( A_M \cap \{ \det(\mathbb H_M^J(\widehat \vartheta_M^J)) \neq 0 \} \right) &\ge \Pr \left( A_M \cap B_M \right) \ge \Pr(A_M) + \Pr(B_M) - 1 \ge p_M - \epl_M,
\end{aligned}
\end{equation}
where the last term tends to one as $M \to \infty$, and which gives the desired result.
\end{proof}

We now consider the zero of the actual estimating function $G_{M,N}^J(\theta)$ and we first analyze its limit as $M \to \infty$.

\begin{lemma} \label{lem:N_MtoInf}
Let the assumptions of \cref{thm:main_unbiased} hold. Then, there exists $N_0 > 0$ such that for all $N > N_0$ an estimator $\widehat \theta_{M,N}^J$, which solves the system $G_{M,N}^J(\theta) = 0$, exists with a probability tending to one as $M$ goes to infinity. Moreover, there exist $\vartheta_N^J$ solving $\mathcal G_N^J(\theta) = 0$ such that
\begin{equation}
\lim_{M \to \infty} \widehat \theta_{M,N}^J = \vartheta_N^J, \qquad \text{in probability},  
\end{equation}
and 
\begin{equation}
\lim_{M \to \infty} \sqrt{M} \left( \widehat \theta_{M,N}^J - \vartheta_N^J \right) = \Lambda_N^J \sim \mathcal N(\mathbf 0, \Gamma_N^J), \qquad \text{in distribution},
\end{equation}
where $\Gamma_N^J$ is a positive definite covariance matrix such that $\lim_{N \to \infty} \Gamma_N^J = \Gamma_0^J$ where $\Gamma_0^J$ is defined in \eqref{eq:variance_normal}.
\end{lemma}
\begin{proof}
First, by \cref{lem:NtoInf} there exists $N_0 > 0$ such that for all $N > N_0$ there exists $\vartheta_N^J$ such that
\begin{equation}
\mathcal G_N^J(\vartheta_N^J) = 0 \qquad \text{and} \qquad \det (\mathcal H_N^J(\vartheta_N^J)) \neq 0.
\end{equation}
Then, the results are equivalent to \cref{lem:MtoInf} and therefore the argument follows the same steps of its proof, which is given in detail in \cite[Theorem 3.2]{BiS95} and is based on \cite[Theorem A.1]{BaS94}. Finally, the convergence of the covariance matrix $\Gamma_N^J$ is implied by \eqref{eq:convergence2meanFieldLimit}.
\end{proof}

We then study the limit of the zero of $G_{M,N}^J(\theta)$ as $N \to \infty$.

\begin{lemma} \label{lem:M_NtoInf}
Let the assumptions of \cref{lem:N_MtoInf} hold and let $M \ll N$. Then, the estimator $\widehat \theta_{M,N}^J$ satisfies for some $\widehat \vartheta_M^J$ solving $\mathbb G_M^J(\theta) = 0$ and for a constant $C > 0$ independent of $M$ and $N$
\begin{equation} \label{eq:difference_thetaMN}
\E \left[ \norm{\widehat \theta_{M,N}^J - \widehat \vartheta_M^J} \right] \le C \sqrt{\frac{M}{N}}.
\end{equation}
\end{lemma}
\begin{proof}
The existence of the estimators $\widehat \vartheta_M^J$, such that $\mathbb G_M^J(\widehat \vartheta_M^J) = 0$ and $\det (\mathbb H_M^J(\widehat \vartheta_M^J)) \neq 0$, and $\widehat \theta_{M,N}^J$, such that $G_{M,N}^J(\widehat \theta_{M,N}^J) = 0$, with a probability tending to one as $M$ goes to infinity is guaranteed by \cref{lem:N_MtoInf,lem:MtoInf}, respectively. Then, all the following events are considered as conditioned on the existence of $\widehat \vartheta_M^J$ and $\widehat \theta_{M,N}^J$ and the fact that $\det (\mathbb H_M^J(\widehat \vartheta_M^J)) \neq 0$. Let us now define the function $f \colon \R^p \times \R^{M+1} \to \R^p$ as
\begin{equation}
f(\theta, x) = \frac1M \sum_{m=0}^{M-1} \sum_{j=1}^J g_j(x_m, x_{m+1}; \theta),
\end{equation}
where $x_m$ denotes the $m$-th component of the vector $x \in \R^{m+1}$, and the vectors $\mathbb X^{(n)}$ and $\mathbb X$ whose $m$-th components for $m = 0, \dots, M$ are given by
\begin{equation}
\mathbb X_m^{(n)} = \widetilde X_m^{(n)} \qquad \text{and} \qquad \mathbb X_m = \widetilde X_m,
\end{equation}
where $\{ \widetilde X_m^{(n)} \}_{m=0}^M$ is the set of observations and $\{ \widetilde X_m \}_{m=0}^M$ are the corresponding realizations of the mean field limit. Notice that $f \in C^1(\Theta \times \R^{M+1})$ due to \cref{ass:functions_psi,rem:continuityG} and by definition we have
\begin{equation}
f(\widehat \vartheta_M^J, \mathbb X) = 0 \qquad \text{and} \qquad \det \left( \frac{\partial f}{\partial \theta}(\widehat \vartheta_M^J, \mathbb X) \right) \neq 0.
\end{equation}
Therefore, applying the implicit function theorem there exist $\epl,\delta > 0$ and a continuously differentiable function $F \colon B_\epl(\mathbb X) \to B_\delta(\widehat \vartheta_M^J)$ such that $f(F(x),x) = 0$ for all $x \in B_\epl(\mathbb X)$. Hence, if $\mathbb X^{(n)}$ is close enough to $\mathbb X$ then there must be one $\widehat \theta_{M,N}^J \in B_\delta(\widehat \vartheta_M^J)$ such that $F(\mathbb X^{(n)}) = \widehat \theta_{M,N}^J$. Then, employing Jensen's inequality and by estimate \eqref{eq:convergence2meanFieldLimit} we have
\begin{equation}
\E \left[ \norm{\mathbb X^{(n)} - \mathbb X} \right] = \E \left[ \left( \sum_{m=0}^M \abs{ \widetilde X_m^{(n)} - \widetilde X_m}^2 \right)^{1/2} \right] \le \left( \sum_{m=0}^M \E \left[ \abs{ \widetilde X_m^{(n)} - \widetilde X_m}^2 \right] \right)^{1/2} \le C \sqrt{\frac{M}{N}},
\end{equation}
where the constant $C$ is independent of $M$ and $N$. Therefore, letting $\epl > 0$ and applying Markov's inequality we obtain
\begin{equation} \label{eq:bound_complement}
\Pr \left( \norm{\mathbb X^{(n)} - \mathbb X} \ge \epl \right) \le \frac1\epl \E \left[ \norm{\mathbb X^{(n)} - \mathbb X} \right] \le \frac{C}{\epl} \sqrt{\frac{M}{N}}.
\end{equation} 
Defining the event $A = \{ \norm{\mathbb X^{(n)} - \mathbb X} < \epl \}$ and using the law of total expectation conditioning on $A$ we deduce
\begin{equation}
\E \left[ \norm{\widehat \theta_{M,N}^J - \widehat \vartheta_M^J} \right] = \E \left[ \norm{\widehat \theta_{M,N}^J - \widehat \vartheta_M^J} | A \right] \Pr(A) + \E \left[ \norm{\widehat \theta_{M,N}^J - \widehat \vartheta_M^J} | A^\compl \right] \Pr(A^\compl),
\end{equation}
which since $\widehat \theta_{M,N}^J, \widehat \vartheta_M^J \in \Theta$, a compact set, and due to estimate \eqref{eq:bound_complement} implies
\begin{equation} \label{eq:total_expectation}
\E \left[ \norm{\widehat \theta_{M,N}^J - \widehat \vartheta_M^J} \right] \le \E \left[ \norm{\widehat \theta_{M,N}^J - \widehat \vartheta_M^J} | A \right] + C \sqrt{\frac{M}{N}}.
\end{equation}
It now remains to study the first term in the right-hand side. Applying the mean value theorem we obtain
\begin{equation}
\begin{aligned}
\mathbb G_M^J(\widehat \theta_{M,N}^J) - G^J_{M,N}(\widehat \theta_{M,N}^J) &= \mathbb G_M^J(\widehat \theta_{M,N}^J) - \mathbb G_M^J(\widehat \vartheta_M^J) \\
&= \left( \int_0^1 \mathbb H_M^J(\widehat \vartheta_M^J + t(\widehat \theta_{M,N}^J - \widehat \vartheta_M^J)) \dd t \right) (\widehat \theta_{M,N}^J - \widehat \vartheta_M^J),
\end{aligned}
\end{equation}
which implies
\begin{equation}
\norm{\widehat \theta_{M,N}^J - \widehat \vartheta_M^J} \le \norm{\left( \int_0^1 \mathbb H_M^J(\widehat \vartheta_M^J + t(\widehat \theta_{M,N}^J - \widehat \vartheta_M^J)) \dd t \right)^{-1}} \norm{\mathbb G_M^J(\widehat \theta_{M,N}^J) - G^J_{M,N}(\widehat \theta_{M,N}^J)}.
\end{equation}
Using Hölder inequality with exponents $q \in (1,2)$ and its conjugate $q'$ such that $1/q + 1/q' = 1$ we have
\begin{equation} \label{eq:decomposition_final}
\E \left[ \norm{\widehat \theta_{M,N}^J - \widehat \vartheta_M^J} | A \right] \le Q \E \left[ \norm{\mathbb G_M^J(\widehat \theta_{M,N}^J) - G^J_{M,N}(\widehat \theta_{M,N}^J)}^q | A \right]^{1/q},
\end{equation}
where
\begin{equation}
Q = \E \left[ \norm{\left( \int_0^1 \mathbb H_M^J(\widehat \vartheta_M^J + t(\widehat \theta_{M,N}^J - \widehat \vartheta_M^J)) \dd t \right)^{-1}}^{q'} | A \right]^{1/q'}.
\end{equation}
Employing the inequality $\E[Y|A] \le \E[Y]/\Pr(A)$, which holds for any positive random variable $Y$, point $(i)'$ in \cref{pro:limitsG} and estimate \eqref{eq:bound_complement}, the second term in the right-hand side can be bounded by
\begin{equation} \label{eq:bound_afterQ}
\E \left[ \norm{\mathbb G_M^J(\widehat \theta_{M,N}^J) - G^J_{M,N}(\widehat \theta_{M,N}^J)}^q | A \right]^{1/q} \le \frac{C}{\sqrt N} \left( \frac{1}{1 - C \sqrt{\frac{M}{N}}} \right)^{1/q} \le \frac{C}{\sqrt N},
\end{equation}
where the last inequality is justified by the fact that $M \ll N$ and by changing the value of the constant $C$. We now have to bound the first term $Q$ in the right-hand side of equation \eqref{eq:decomposition_final}. Employing the inequality $\norm{M^{-1}} \le \norm{M}^{p-1}/\abs{\det(M)}$, which holds for any square nonsingular matrix $M \in \R^{p \times p}$, we have
\begin{equation}
Q \le \E \left[ \frac{\norm{\int_0^1 \mathbb H_M^J(\widehat \vartheta_M^J + t(\widehat \theta_{M,N}^J - \widehat \vartheta_M^J)) \dd t}^{q'(p-1)}}{\abs{\det \left( \int_0^1 \mathbb H_M^J(\widehat \vartheta_M^J + t(\widehat \theta_{M,N}^J - \widehat \vartheta_M^J)) \dd t \right)}^{q'}} | A \right].
\end{equation}
Since we are conditioning on the event $A$, by the first part of the proof, we know that $\norm{\widehat \theta_{M,N}^J - \widehat \vartheta_M^J} \le \delta$ and, by taking $\epl$ sufficiently small, we can always find $\delta$ small enough, but still finite, such that the absolute value of the determinant in the denominator is lower bounded by a constant independent of $M$ and $N$ because $\det (\mathbb H_M^J(\widehat \vartheta_M^J)) \neq 0$ and by \eqref{eq:convergence_Jacobian} it converges in probability to $\det(\mathscr H^J(\theta_0))$, which is invertible. Hence, applying Jensen's inequality we obtain
\begin{equation}
\begin{aligned}
Q &\le C \E \left[ \norm{\int_0^1 \mathbb H_M^J(\widehat \vartheta_M^J + t(\widehat \theta_{M,N}^J - \widehat \vartheta_M^J)) \dd t}^{q'(p-1)} | A \right] \\
&\le C \E \left[ \int_0^1 \norm{\mathbb H_M^J(\widehat \vartheta_M^J + t(\widehat \theta_{M,N}^J - \widehat \vartheta_M^J))}^{q'(p-1)} \dd t | A \right],
\end{aligned}
\end{equation}
which due to \cref{lem:boundedG}, \cref{rem:convergence_derivatives}, the property $\E[Y|A] \le \E[Y]/\Pr(A)$, which holds for any positive random variable $Y$, and estimate \eqref{eq:bound_complement} yields
\begin{equation}
Q \le \frac{C}{\Pr(A)} \int_0^1 \E \left[ \norm{\mathbb H_M^J(\widehat \vartheta_M^J + t(\widehat \theta_{M,N}^J - \widehat \vartheta_M^J))}^{q'(p-1)} \right] \dd t \le C,
\end{equation}
which together with equations \eqref{eq:total_expectation}, \eqref{eq:decomposition_final} and \eqref{eq:bound_afterQ} gives the desired result.
\end{proof}

The results of this section are summarized in the following diagram
\begin{center}
\begin{tikzpicture}
\node at (-5,0) (W) {$\widehat \theta_{M,N}^J$};
\node at (0,1) (N) {$\widehat \vartheta_M^J$};
\node at (0,-1) (S) {$\vartheta_N^J$};
\node at (5,0) (E) {$\theta_0$};
\draw [->] (W) -- (N);
\node at (-2.5,0.75) {\emph{in $L^1$}};
\node at (-2,0.25) {$N \to \infty$};
\draw [->] (N) -- (E);
\node at (2.5,0.75) {\emph{in $\Pr$}};
\node at (2,0.25) {$M \to \infty$};
\draw [->] (W) -- (S);
\node at (-2.5,-0.75) {\emph{in $\Pr$}};
\node at (-2,-0.25) {$M \to \infty$};
\draw [->] (S) -- (E);
\node at (2,-0.25) {$N \to \infty$};
\end{tikzpicture}
\end{center}
where $\Pr$ stands for convergence in probability.

\begin{remark}
All the previous results only prove the existence of such estimators with high probability and do not guarantee their uniqueness. However, as we will see in the next section, any of these estimators converge to the exact value of the unknown.
\end{remark}

\subsection{Proof of the main theorems}

In this section we finally present the proofs of the main results of this work, i.e., \cref{thm:main_unbiased,thm:main_rate,thm:main_normal}.

\begin{proof}[Proof of \cref{thm:main_unbiased}]
First, by \cref{lem:N_MtoInf} we deduce the existence of $N_0 > 0$ such that for all $N > N_0$ the estimator $\widehat \theta_{M,N}^J$ exists with a probability tending to one as $M$ goes to infinity. Then, we prove separately equations \eqref{eq:limit_1}, \eqref{eq:limit_2} and \eqref{eq:limit_3}. \\
\textbf{Proof of \eqref{eq:limit_1}.} By \cref{lem:N_MtoInf,lem:NtoInf} we have
\begin{equation}
\lim_{N \to \infty} \lim_{M \to \infty} \widehat \theta_{M,N}^J = \lim_{N \to \infty} \vartheta_N^J = \theta_0, \qquad \text{in probability},
\end{equation}
which proves \eqref{eq:limit_1}. \\
\textbf{Proof of \eqref{eq:limit_2}.} By \cref{lem:M_NtoInf} the estimator $\widehat \theta_{M,N}^J$ converges to $\widehat \vartheta_M^J$ in $L^1$ as $N$ goes to infinity and hence in probability. Therefore, applying \cref{lem:MtoInf} we obtain
\begin{equation}
\lim_{M \to \infty} \lim_{N \to \infty} \widehat \theta_{M,N}^J = \lim_{M \to \infty} \widehat \vartheta_M^J = \theta_0, \qquad \text{in probability},
\end{equation}
which shows \eqref{eq:limit_2}. \\
\textbf{Proof of \eqref{eq:limit_3}.} We introduce the following decomposition
\begin{equation}
\widehat \theta_{M,N}^J - \theta_0 = (\widehat \theta_{M,N}^J - \widehat \vartheta_M^J) + (\widehat \vartheta_M^J - \theta_0) \eqdef Q_1 + Q_2,
\end{equation}
where $\widehat \vartheta_M^J$ is defined in \cref{lem:MtoInf} and due to \cref{lem:M_NtoInf} the first quantity satisfies
\begin{equation} \label{eq:Q1_N}
\E \left[ \abs{Q_1} \right] \le C \sqrt{\frac{M}{N}},
\end{equation}
with the constant $C$ independent of $M$ and $N$. Therefore, since $M = o(N)$, estimate \eqref{eq:Q1_N} together with \cref{lem:MtoInf} and the fact that convergence in $L^1$ implies convergence in probability gives the desired result \eqref{eq:limit_3} and ends the proof.
\end{proof}

\begin{proof}[Proof of \cref{thm:main_rate}]
The existence of the estimator $\widehat \theta_{M,N}^J$ is given by \cref{thm:main_unbiased}. Then, we prove separately equations \eqref{eq:rate_1}, \eqref{eq:rate_2} and \eqref{eq:rate_3}. \\
\textbf{Proof of \eqref{eq:rate_1}.} Let $\vartheta_N$ be defined in \cref{lem:NtoInf}. Using basic properties of probability measures we have
\begin{equation} \label{eq:probability_inequalities}
\begin{aligned}
\Pr \left( \Xi_{M,N}^J > K_\epl \right) &= \Pr \left( \norm{\widehat \theta_{M,N}^J - \theta_0} > \left( \frac{1}{\sqrt M} + \frac{1}{\sqrt N} \right) K_\epl \right) \\
&\le \Pr \left( \norm{\widehat \theta_{M,N}^J - \vartheta_N} + \norm{\vartheta_N - \theta_0} > \left( \frac{1}{\sqrt M} + \frac{1}{\sqrt N} \right) K_\epl \right),
\end{aligned}
\end{equation}
which implies
\begin{equation}
\begin{aligned}
\Pr \left( \Xi_{M,N}^J > K_\epl \right) &\le \Pr \left( \norm{\widehat \theta_{M,N}^J - \vartheta_N} > \left( \frac{1}{\sqrt M} + \frac{1}{\sqrt N} \right) \frac{K_\epl}{2} \right) \\
&\quad + \Pr \left( \norm{\vartheta_N - \theta_0} > \left( \frac{1}{\sqrt M} + \frac{1}{\sqrt N} \right) \frac{K_\epl}{2} \right) \\
&\le \Pr \left( \sqrt M \norm{\widehat \theta_{M,N}^J - \vartheta_N} > \frac{K_\epl}{2} \right) + \Pr \left( \norm{\vartheta_N - \theta_0} > \frac{K_\epl}{2 \sqrt N} \right),
\end{aligned}
\end{equation}
and we now study two terms in the right-hand side separately. First, letting $M$ and $N$ go to infinity by \cref{lem:N_MtoInf} we obtain
\begin{equation}
\lim_{N \to \infty} \lim_{M \to \infty} \Pr \left( \sqrt M \norm{\widehat \theta_{M,N}^J - \vartheta_N} > \frac{K_\epl}{2} \right) = \Pr \left( \norm{\Lambda^J} > \frac{K_\epl}{2} \right),
\end{equation}
where the right-hand side can be made arbitrarily small by taking $K_\epl > 0$ sufficiently large. Moreover, we have
\begin{equation}
\Pr \left( \norm{\vartheta_N - \theta_0} > \frac{K_\epl}{2 \sqrt N} \right) = \E \left[ \mathbbm 1_{\left\{ \norm{\vartheta_N - \theta_0} > \frac{K_\epl}{2 \sqrt N} \right\}} \right],
\end{equation}
where the right-hand side is identically equal to zero if we set $K_\epl > 2C$, where the constant $C$ is given by \cref{lem:NtoInf}. Hence, for all $\epl > 0$ we can take $K_\epl > 0$ sufficiently large such that 
\begin{equation}
\lim_{N \to \infty} \lim_{M \to \infty} \Pr \left( \Xi_{M,N}^J > K_\epl \right) < \epl,
\end{equation}
which proves \eqref{eq:rate_1}. \\
\textbf{Proof of \eqref{eq:rate_2}.} Let $\widehat \vartheta_M$ be defined in \cref{lem:MtoInf}. Repeating a procedure similar to \eqref{eq:probability_inequalities} and applying Markov's inequality we get
\begin{equation}
\begin{aligned}
\Pr \left( \Xi_{M,N}^J > K_\epl \right) &\le \Pr \left( \norm{\widehat \theta_{M,N}^J - \widehat \vartheta_M} > \left( \frac{1}{\sqrt M} + \frac{1}{\sqrt N} \right) \frac{K_\epl}{2} \right) + \Pr \left( \sqrt M \norm{\widehat \vartheta_M - \theta_0} > \frac{K_\epl}{2} \right) \\
&\le \frac{2 \sqrt{MN}}{K_\epl (\sqrt M + \sqrt N)} \E \left[ \norm{\widehat \theta_{M,N}^J - \widehat \vartheta_M} \right] + \Pr \left( \sqrt M \norm{\widehat \vartheta_M - \theta_0} > \frac{K_\epl}{2} \right),
\end{aligned}
\end{equation}
and we now study two terms in the right-hand side separately. First, by \cref{lem:MtoInf} we have
\begin{equation}
\lim_{M \to \infty} \Pr \left( \sqrt M \norm{\widehat \vartheta_M - \theta_0} > \frac{K_\epl}{2} \right) = \Pr \left( \norm{\Lambda^J} > \frac{K_\epl}{2} \right),
\end{equation}
where the right-hand side can be made arbitrarily small by taking $K_\epl > 0$ sufficiently large. Moreover, by \cref{lem:M_NtoInf} we have
\begin{equation} \label{eq:rate_lastLine}
\frac{2 \sqrt{MN}}{K_\epl (\sqrt M + \sqrt N)} \E \left[ \norm{\widehat \theta_{M,N}^J - \widehat \vartheta_M} \right] \le \frac{2CM}{K_\epl (\sqrt M + \sqrt N)},
\end{equation}
where the constant $C$ is independent of $M$ and $N$. Hence, for all $\epl > 0$ we can take $K_\epl > 0$ sufficiently large such that 
\begin{equation}
\lim_{M \to \infty} \lim_{N \to \infty} \Pr \left( \Xi_{M,N}^J > K_\epl \right) < \epl,
\end{equation}
which shows \eqref{eq:rate_2}. \\
\textbf{Proof of \eqref{eq:rate_3}.} Equation \eqref{eq:rate_3} is obtained following verbatim the proof of \eqref{eq:rate_2} in the previous step and using the fact that $M = o(\sqrt N)$ to show that the right-hand side in equation \eqref{eq:rate_lastLine} vanishes.
\end{proof}

\begin{proof}[Proof of \cref{thm:main_normal}]
The existence of the estimator $\widehat \theta_{M,N}^J$ is given by \cref{thm:main_unbiased}. Then, let us introduce the following decomposition 
\begin{equation}
\sqrt M \left( \widehat \theta_{M,N}^J - \theta_0 \right) = \sqrt M \left( \widehat \theta_{M,N}^J - \widehat \vartheta_M^J \right) + \sqrt M \left( \widehat \vartheta_M^J - \theta_0 \right),
\end{equation}
where $\widehat \vartheta_M^J$ is defined in \cref{lem:MtoInf}. We now study two terms in the right-hand side separately. By \cref{lem:M_NtoInf} we have
\begin{equation}
\sqrt M \E \left[ \norm{\widehat \theta_{M,N}^J - \widehat \vartheta_M^J} \right] \le C \frac{M}{\sqrt N},
\end{equation}
where the constant $C$ is independent of $M$ and $N$, hence since $M = o(\sqrt N)$ by hypothesis we obtain
\begin{equation} \label{eq:normal_limit_1}
\lim_{M,N \to \infty} \sqrt M \left( \widehat \theta_{M,N}^J - \widehat \vartheta_M^J \right) = 0, \qquad \text{in probability}.
\end{equation}
Moreover, by \cref{lem:MtoInf} we know that 
\begin{equation} \label{eq:normal_limit_2}
\lim_{M \to \infty} \sqrt M \left( \widehat \vartheta_M^J - \theta_0 \right) = \Lambda^J \sim \mathcal N(0, \Gamma_0^J), \qquad \text{in distribution},
\end{equation}
where the covariance matrix $\Gamma_0^J$ is defined in \eqref{eq:variance_normal}. Finally, limits \eqref{eq:normal_limit_1} and \eqref{eq:normal_limit_2} together with Slutsky's theorem imply the desired result.
\end{proof}

\section{Conclusion}

In this work we considered inference problems for large systems of exchangeable interacting particles. When the number of  particles is large, then the path of a single particle is well approximated by its mean field limit. The limiting mean field SDE is on the one hand more complex because it is a nonlinear SDE (in the sense of McKean), but on the other hand more tractable from a computational viewpoint as it reduces an $N$-dimensional SDE to a one dimensional one. Our aim was to infer unknown parameters of the dynamics, in particular of the confining and interaction potentials, from a set of discrete observations of a single particle. We propose a novel estimator which is obtained by computing the zero of a martingale estimating function based on the eigenvalues and the eigenfunctions of the generator of the mean field limit, where the law of the process is replaced by the (unique) invariant measure of the mean field dynamics. We showed both theoretically and numerically the asymptotic unbiasedness and normality of our estimator in the limit of infinite data and particles, providing also a rate of convergence towards the true value of the unknown parameter. In particular, we observed that these properties hold true if the number of particles is much larger than the number of observations. Even though our theoretical results require uniqueness of the steady state for the mean field dynamics, our numerical experiments suggest that our method works well even when phase transitions are present, i.e., when there are more than one stationary states. Moreover, we compared our estimator with the maximum likelihood estimator, demonstrating that our approach is more robust with respect to small values of the sampling rate. We believe, therefore, that the inference methodology proposed and analyzed in this paper can be very efficient when learning parameters in mean field SDE models from data. 

The work presented in this paper can be extended in several interesting directions. First, the main limitation of our methodology is the fact that in order to construct the martingale estimating function we have to know the functional form of the invariant measure of the mean field SDE, possibly parameterized in terms of a finite number of moments. There are many interesting examples of mean field PDEs where the self-consistency equation cannot be solved analytically or, at least, its solution depends on the unknown parameters in the model. Therefore, it would be interesting to lift this assumption by first learning the invariant measure from data and then applying our martingale eigenfunction estimator approach. This leads naturally to our second objective, namely the extension of our methodology to a nonparametric setting, i.e., when the functional form of the confining and interaction potentials are unknown. Thirdly, we want to obtain more detailed information on the computational complexity of the proposed algorithm, in particular when more eigenfunctions are needed for our martingale estimator and when we are in higher dimensions in space. We will return to these problems in future work.

\subsection*{Acknowledgements} 

The work of GAP was partially funded by the EPSRC, grant number EP/P031587/1, and by JPMorgan Chase \& Co. AZ is partially supported by the Swiss National Science Foundation, under grant No. 200020\_172710. AZ is grateful to Massimo Sorella for fruitful discussions.

\bibliographystyle{siamnodash}
\bibliography{biblio}

\begin{thebibliography}{10}

\bibitem{AGP21}
{\sc A.~Abdulle, G.~Garegnani, G.~Pavliotis, A.~Stuart, and A.~Zanoni}, {\em
  Drift estimation of multiscale diffusions based on filtered data}, Found.
  Comput. Math.,  (2021), pp.~1--52.

\bibitem{APZ21}
{\sc A.~Abdulle, G.~A. Pavliotis, and A.~Zanoni}, {\em Eigenfunction martingale
  estimating functions and filtered data for drift estimation of discretely
  observed multiscale diffusions}.
\newblock Preprint arXiv:2104.10587, 2021.

\bibitem{BaS94}
{\sc O.~E. Barndorff-Nielsen and M.~Sørensen}, {\em A review of some aspects
  of asymptotic likelihood theory for stochastic processes}, International
  Statistical Review / Revue Internationale de Statistique, 62 (1994),
  pp.~133--165.

\bibitem{BRT98}
{\sc S.~Benachour, B.~Roynette, D.~Talay, and P.~Vallois}, {\em Nonlinear
  self-stabilizing processes. {I}. {E}xistence, invariant probability,
  propagation of chaos}, Stochastic Process. Appl., 75 (1998), pp.~173--201.

\bibitem{BRV98}
{\sc S.~Benachour, B.~Roynette, and P.~Vallois}, {\em Nonlinear
  self-stabilizing processes. {II}. {C}onvergence to invariant probability},
  Stochastic Process. Appl., 75 (1998), pp.~203--224.

\bibitem{BiS95}
{\sc B.~Bibby and M.~S\o~rensen}, {\em Martingale estimation functions for
  discretely observed diffusion processes}, Bernoulli, 1 (1995), pp.~17--39.

\bibitem{BiT08}
{\sc J.~Binney and S.~Tremaine}, {\em Galactic Dynamics}, Princeton University
  Press, Princeton, second~ed., 2008.

\bibitem{Bis11}
{\sc J.~P.~N. Bishwal}, {\em Estimation in interacting diffusions: continuous
  and discrete sampling}, Appl. Math. (Irvine), 2 (2011), pp.~1154--1158.

\bibitem{CGP20}
{\sc J.~A. Carrillo, R.~S. Gvalani, G.~A. Pavliotis, and A.~Schlichting}, {\em
  Long-time behaviour and phase transitions for the {M}ckean-{V}lasov equation
  on the torus}, Arch. Ration. Mech. Anal., 235 (2020), pp.~635--690.

\bibitem{Che21}
{\sc X.~Chen}, {\em Maximum likelihood estimation of potential energy in
  interacting particle systems from single-trajectory data}, Electron. Commun.
  Probab., 26 (2021), pp.~Paper No. 45, 13.

\bibitem{Daw83}
{\sc D.~A. Dawson}, {\em Critical dynamics and fluctuations for a mean-field
  model of cooperative behavior}, J. Statist. Phys., 31 (1983), pp.~29--85.

\bibitem{DeT21}
{\sc F.~Delarue and A.~Tse}, {\em Uniform in time weak propagation of chaos on
  the torus}, 2021.

\bibitem{DGP21}
{\sc M.~G. Delgadino, R.~S. Gvalani, and G.~A. Pavliotis}, {\em On the
  {D}iffusive-{M}ean {F}ield {L}imit for {W}eakly {I}nteracting {D}iffusions
  {E}xhibiting {P}hase {T}ransitions}, Arch. Ration. Mech. Anal., 241 (2021),
  pp.~91--148.

\bibitem{EGP18}
{\sc L.~Ellam, M.~Girolami, G.~A. Pavliotis, and A.~Wilson}, {\em Stochastic
  modelling of urban structure}, Proc. A., 474 (2018), pp.~20170700, 20.

\bibitem{Fra05}
{\sc T.~D. Frank}, {\em Nonlinear {F}okker-{P}lanck equations}, Springer Series
  in Synergetics, Springer-Verlag, Berlin, 2005.
\newblock Fundamentals and applications.

\bibitem{Gan08}
{\sc A.~Ganz~Bustos}, {\em Approximations des distributions d'{\'e}quilibre de
  certains syst{\`e}mes stochastiques avec interactions McKean-Vlasov}, PhD
  thesis, Nice, 2008.

\bibitem{GaZ21}
{\sc G.~Garegnani and A.~Zanoni}, {\em Robust estimation of effective
  diffusions from multiscale data}.
\newblock Preprint arXiv:2109.03132, 2021.

\bibitem{GPY17}
{\sc J.~Garnier, G.~Papanicolaou, and T.-W. Yang}, {\em Consensus convergence
  with stochastic effects}, Vietnam J. Math., 45 (2017), pp.~51--75.

\bibitem{Gar88}
{\sc J.~G\"{a}rtner}, {\em On the {M}c{K}ean-{V}lasov limit for interacting
  diffusions}, Math. Nachr., 137 (1988), pp.~197--248.

\bibitem{GSS20}
{\sc K.~Giesecke, G.~Schwenkler, and J.~A. Sirignano}, {\em Inference for large
  financial systems}, Math. Finance, 30 (2020), pp.~3--46.

\bibitem{GGS21}
{\sc B.~D. Goddard, B.~Gooding, H.~Short, and G.~A. Pavliotis}, {\em {Noisy
  bounded confidence models for opinion dynamics: the effect of boundary
  conditions on phase transitions}}, IMA Journal of Applied Mathematics,
  (2021).

\bibitem{Gol16}
{\sc F.~Golse}, {\em On the dynamics of large particle systems in the mean
  field limit}, in Macroscopic and large scale phenomena: coarse graining, mean
  field limits and ergodicity, vol.~3 of Lect. Notes Appl. Math. Mech.,
  Springer, [Cham], 2016, pp.~1--144.

\bibitem{GoP18}
{\sc S.~N. Gomes and G.~A. Pavliotis}, {\em Mean field limits for interacting
  diffusions in a two-scale potential}, J. Nonlinear Sci., 28 (2018),
  pp.~905--941.

\bibitem{GSW19}
{\sc S.~N. Gomes, A.~M. Stuart, and M.-T. Wolfram}, {\em Parameter estimation
  for macroscopic pedestrian dynamics models from microscopic data}, SIAM J.
  Appl. Math., 79 (2019), pp.~1475--1500.

\bibitem{HST98}
{\sc L.~P. Hansen, J.~A. Scheinkman, and N.~Touzi}, {\em Spectral methods for
  identifying scalar diffusions}, J. Econometrics, 86 (1998), pp.~1--32.

\bibitem{Kas90}
{\sc R.~A. Kasonga}, {\em Maximum likelihood theory for large interacting
  systems}, SIAM J. Appl. Math., 50 (1990), pp.~865--875.

\bibitem{KeS99}
{\sc M.~Kessler and M.~S\o~rensen}, {\em Estimating equations based on
  eigenfunctions for a discretely observed diffusion process}, Bernoulli, 5
  (1999), pp.~299--314.

\bibitem{Lan93}
{\sc S.~Lang}, {\em Real and functional analysis}, vol.~142 of Graduate Texts
  in Mathematics, Springer-Verlag, New York, third~ed., 1993.

\bibitem{LiQ20}
{\sc M.~Liu and H.~Qiao}, {\em Parameter estimation of path-dependent
  {M}c{K}ean-{V}lasov stochastic differential equations}, 2020.

\bibitem{MKT18}
{\sc N.~K. Mahato, A.~Klar, and S.~Tiwari}, {\em Particle methods for
  multi-group pedestrian flow}, Appl. Math. Model., 53 (2018), pp.~447--461.

\bibitem{Mal01}
{\sc F.~Malrieu}, {\em Logarithmic {S}obolev inequalities for some nonlinear
  {PDE}'s}, Stochastic Process. Appl., 95 (2001), pp.~109--132.

\bibitem{MaF19}
{\sc B.~Maury and S.~Faure}, {\em Crowds in equations}, Advanced Textbooks in
  Mathematics, World Scientific Publishing Co. Pte. Ltd., Hackensack, NJ, 2019.
\newblock An introduction to the microscopic modeling of crowds, With a
  foreword by Laure Saint-Raymond.

\bibitem{Oel84}
{\sc K.~Oelschl\"{a}ger}, {\em A martingale approach to the law of large
  numbers for weakly interacting stochastic processes}, Ann. Probab., 12
  (1984), pp.~458--479.

\bibitem{PaS95}
{\sc C.~D. Pagani and S.~Salsa}, {\em Analisi matematica - Volume 1}, Masson,
  1995.

\bibitem{PPS09}
{\sc A.~Papavasiliou, G.~A. Pavliotis, and A.~M. Stuart}, {\em Maximum
  likelihood drift estimation for multiscale diffusions}, Stochastic Process.
  Appl., 119 (2009), pp.~3173--3210.

\bibitem{Pav14}
{\sc G.~A. Pavliotis}, {\em Stochastic processes and applications}, vol.~60 of
  Texts in Applied Mathematics, Springer, New York, 2014.
\newblock Diffusion processes, the Fokker-Planck and Langevin equations.

\bibitem{PaS07}
{\sc G.~A. Pavliotis and A.~M. Stuart}, {\em Parameter estimation for
  multiscale diffusions}, J. Stat. Phys., 127 (2007), pp.~741--781.

\bibitem{Sch74}
{\sc T.~B. Scheffler}, {\em Analyticity of the eigenvalues and eigenfunctions
  of an ordinary differential operator with respect to a parameter}, Proc. Roy.
  Soc. London Ser. A, 336 (1974), pp.~475--486.

\bibitem{SKP21}
{\sc L.~Sharrock, N.~Kantas, P.~Parpas, and G.~A. Pavliotis}, {\em Parameter
  estimation for the {M}c{K}ean-{V}lasov stochastic differential equation},
  2021.

\bibitem{Suz05}
{\sc T.~Suzuki}, {\em Free energy and self-interacting particles}, vol.~62 of
  Progress in Nonlinear Differential Equations and their Applications,
  Birkh\"{a}user Boston, Inc., Boston, MA, 2005.

\bibitem{Szn91}
{\sc A.-S. Sznitman}, {\em Topics in propagation of chaos}, in \'{E}cole
  d'\'{E}t\'{e} de {P}robabilit\'{e}s de {S}aint-{F}lour {XIX}---1989,
  vol.~1464 of Lecture Notes in Math., Springer, Berlin, 1991, pp.~165--251.

\bibitem{Tug14}
{\sc J.~Tugaut}, {\em Phase transitions of {M}c{K}ean-{V}lasov processes in
  double-wells landscape}, Stochastics, 86 (2014), pp.~257--284.

\bibitem{WWM16}
{\sc J.~Wen, X.~Wang, S.~Mao, and X.~Xiao}, {\em Maximum likelihood estimation
  of {M}c{K}ean-{V}lasov stochastic differential equation and its application},
  Appl. Math. Comput., 274 (2016), pp.~237--246.

\bibitem{Yos03}
{\sc N.~Yoshida}, {\em Phase transition from the viewpoint of relaxation
  phenomena}, Rev. Math. Phys., 15 (2003), pp.~765--788.

\end{thebibliography}

\end{document}